\documentclass[11pt]{amsart}

\usepackage{amsmath, amsthm}
\usepackage{eucal}

\usepackage{amssymb}
\usepackage{amscd}
\usepackage{palatino}
\usepackage{latexsym}
\usepackage{epsfig}
\usepackage{graphicx}
\usepackage{amsfonts}
\usepackage{psfrag}

\usepackage{url}
\usepackage{cite}

\usepackage[margin=1in]{geometry}

\input xy
\xyoption{all}
\UseComputerModernTips

\numberwithin{equation}{section}
\numberwithin{figure}{section}

\newtheorem{theorem}{Theorem}[section]
\newtheorem{lemma}[theorem]{Lemma}
\newtheorem{proposition}[theorem]{Proposition}
\newtheorem{corollary}[theorem]{Corollary}

\newtheorem{fact}[theorem]{Fact}

\newtheorem{remark}[theorem]{Remark}
\newtheorem{example}[theorem]{Example}

\theoremstyle{definition}
\newtheorem{definition}[theorem]{Definition}

\newcommand\toe[2]{\mathcal{T}_{#1}(#2)}
\newcommand\head[2]{\mathcal{H}_{#1}(#2)}

\newcommand{\C}{{\mathbb{C}}}

\newcommand{\Z}{{\mathbb{Z}}}
\newcommand{\Q}{{\mathbb{Q}}}

\newcommand{\F}{\mathbb{F}}

\renewcommand{\t}{\mathfrak{t}}

\newcommand{\into}{\hookrightarrow}

\providecommand{\abs}[1]{\lvert#1\rvert}

\DeclareMathOperator{\Lie}{Lie}

\DeclareMathOperator{\Gr}{Gr}

\DeclareMathOperator{\pt}{pt}

\DeclareMathOperator{\Sym}{Sym}

\newcommand{\hsm}{{\hspace{1mm}}}

\begin{document}

\title{A positive Monk formula in the $S^1$-equivariant cohomology of type $A$ Peterson varieties}

 \author{Megumi Harada}
 \address{Department of Mathematics and
 Statistics\\ McMaster University\\ 1280 Main Street West\\ Hamilton, Ontario L8S4K1\\ Canada}
 \email{Megumi.Harada@math.mcmaster.ca}
 \urladdr{\url{http://www.math.mcmaster.ca/Megumi.Harada/}}
 \thanks{MH is partially supported by an NSERC Discovery Grant,
 an NSERC University Faculty Award, and an Ontario Ministry of Research
 and Innovation Early Researcher Award.  JT is partially supported by 
 NSF grant DMS-0801554, a Sloan Research Fellowship, and an Old Gold Fellowship.}

 \author{Julianna Tymoczko}
 \address{Department of Mathematics, University of Iowa, 14 MacLean
 Hall, Iowa City, Iowa 52242-1419, U.S.A.}
 \email{tymoczko@uiowa.edu}
\urladdr{\url{http://www.math.uiowa.edu/~tymoczko/}}

\keywords{} 

\date{\today}

%%%%%%%%%%%%%%%%%%%%%
%  Abstract
%%%%%%%%%%%%%%%%%%%%%

\begin{abstract}

\textbf{Peterson varieties} are a special class
of Hessenberg varieties that have been extensively 
 studied e.g. by Peterson, Kostant, and Rietsch,
  in connection with the quantum cohomology of the flag variety. 
In this manuscript, we develop 
  a \textit{generalized Schubert calculus}, and in particular a
  \textit{positive} Chevalley-Monk formula, for the ordinary and
  Borel-equivariant cohomology of the Peterson variety
  $Y$ in type $A_{n-1}$, with respect to a natural $S^1$-action
  arising from the standard action of the maximal torus on flag
  varieties. As far as we know, this is the first example of positive
  Schubert calculus beyond the realm of Kac-Moody flag varieties
  $G/P$. 
  
  Our main results are as follows. First,  we identify a computationally convenient basis of
  $H^*_{S^1}(Y)$, which we call the basis of \textbf{Peterson Schubert
    classes}.
 Second, we derive a \textbf{manifestly positive, integral}
  \textbf{Chevalley-Monk formula} for the product of a 
 cohomology-degree-$2$ Peterson Schubert
  class with an arbitrary Peterson Schubert class.
Both $H^*_{S^1}(Y)$ and $H^*(Y)$ are generated in degree $2$.
Finally, by using our Chevalley-Monk formula we give explicit
descriptions (via generators and relations) of both the
$S^1$-equivariant cohomology ring $H^*_{S^1}(Y)$ and the ordinary
cohomology ring $H^*(Y)$ of the type $A_{n-1}$ Peterson variety.
Our methods are both directly from and inspired by those of
  GKM (Goresky-Kottwitz-MacPherson) theory and classical Schubert
  calculus. 
We discuss several open questions and directions for
  future work.
\end{abstract}

\maketitle

\setcounter{tocdepth}{1}
\tableofcontents

\section{Introduction}\label{sec:intro}

The main results of this manuscript are 
\begin{enumerate} 
\item a construction of a
computationally convenient module basis of \textit{Peterson Schubert classes} for the $S^1$-equivariant
cohomology ring with $\C$ coefficients of the Peterson variety $Y$ of Lie
type $A_{n-1}$, obtained as the projections of a suitable subset of the
well-known equivariant Schubert classes in $H^*_T({\mathcal F}\ell
ags(\C^n))$, and 
\item a manifestly-positive and manifestly-integral \footnote{Our structure constants will 
a priori be elements in $\C[t]$, a polynomial ring in one variable
with $\C$ coefficients. In this setting, by ``positive and integral'' we will mean that the structure constants
are polynomials $\sum_i a_i t^i$ with non-negative and integral
coefficients $a_i \in \Z_{\geq 0}$.}
combinatorial Chevalley-Monk formula for
the computation of certain products in $H^*_{S^1}(Y)$, obtained from 
explicit positive formulas for the restrictions of Peterson Schubert
classes to the $S^1$-fixed points of $Y$. 
\end{enumerate}
Moreover, as a straightforward corollary of the above, we also obtain 
\begin{enumerate} 
\item[(3)] an explicit description, via generators and relations, of
  both the $S^1$-equivariant and ordinary cohomology rings of type $A$
  Peterson varieties. 
\end{enumerate} 
As far as we know, 
this is the first example of an explicit, complete, and combinatorial
Schubert calculus computation of equivariant or ordinary cohomology rings outside of the setting of partial flag varieties $G/P$.  
We use techniques
both directly from and motivated by GKM (Goresky-Kottwitz-MacPherson)
theory and the Schubert calculus of flag varieties. We view our
results as the first steps in the development of a generalized
equivariant Schubert calculus for Hessenberg varieties (of which
Peterson varieties are a special case) and, more generally, for
certain subspaces of GKM spaces.

We begin with some background and motivation. 
Hessenberg varieties arise in many areas of mathematics, including
geometric representation theory, numerical
analysis, mathematical physics, combinatorics, and algebraic geometry. 
Their geometry is complicated and subtle: for instance, 
many Hessenberg varieties are singular, and some are not even
pure-dimensional. However, there is a close relationship between
Hessenberg varieties and linear algebra, which allows for explicit
analysis of 
their geometry and their connections to other fields. For instance, in
the special case of Springer varieties
\cite{Spa76, Shi85, Fun03}, their associated cohomology rings carry natural
representations of the symmetric group such that the top-dimensional
cohomology is an irreducible representation. More generally, 
Hessenberg varieties have a paving by affines indexed by certain
Young tableaux; the tableaux determine the dimension of the affines
according to explicit and simple combinatorial conditions
\cite{Tym07b}. 

In this paper, we focus on the special case of \textbf{Peterson
  varieties}, 
the geometry and combinatorics of which are of 
particular interest and are the subject of active current research.
Indeed, Kostant showed that Peterson varieties have a dense subvariety
whose coordinate ring is isomorphic to the quantum cohomology of the
flag variety \cite{Kos96}. Rietsch additionally proved that the
quantum parameters can be realized as principal minors of certain
Toeplitz matrices \cite{Rie03}. Furthermore, these Toeplitz matrices
can be obtained using a particular Schubert
decomposition of the flag variety intersected with the Peterson variety; 
this Schubert decomposition gives a paving by
affines of the Peterson varieties \cite{Tym06, Tym07b}. 
Much is still
unknown about Peterson varieties.  For example, this paper provides the first
general computation (e.g. with generators and relations) of the
ordinary and equivariant cohomology rings of Peterson varieties.

The goals and methods of this manuscript lie within the realm of
\textbf{Schubert calculus} and \textbf{GKM theory}, both of which
focus on explicit, combinatorial computations in (equivariant and
ordinary) cohomology rings. We begin with a brief discussion of
the former. Classical Schubert calculus is the study of the cohomology
ring $H^*(\Gr(k,\C^n))$ of the Grassmannian of $k$-planes in $\C^n$; 
more specifically, it asks for the
structure constants of $H^*(\Gr(k, \C^n))$ with respect
to the \textit{Schubert classes}, which are classes 
corresponding to \textit{Schubert subvarieties} of $\Gr(k,\C^n)$
and also form an additive basis for the ring. These
Schubert classes are also combinatorially natural in the following
sense: in the Borel presentation of $H^*(\Gr(k,\C^n))$ 
as a quotient of a polynomial ring, it is possible to
represent these classes by \textit{Schur polynomials},
which are essential and ubiquitous in e.g. symmetric function
theory. In the setting of $\Gr(k, \C^n)$, it is known
that the structure constants mentioned above are both
\textbf{positive} and \textbf{integral}. 
Thus, a natural and fundamental goal in
classical Schubert calculus is to find and prove formulas for these
structure constants which are \textit{manifestly} positive and
integral (e.g. by counting arguments).

Modern work in Schubert calculus encompasses the study of more general
spaces, such as the generalized Kac-Moody flag varieties, 
as well as more general (ordinary or equivariant)
cohomology theories, such as Borel-equivariant cohomology with various
coefficient rings, ordinary and equivariant quantum cohomology, as
well as ordinary and equivariant $K$-theory and quantum $K$-theory,
among others. The main goal of modern Schubert calculus is still to
prove that the relevant structure constants are positive in a suitable
sense, and thence to 
obtain explicit, elegant, and/or computationally effective
combinatorial
formulas for these constants. 
Recently, efforts have been made to
extend the ideas of Schubert calculus to cover even more general
spaces (e.g. the work of Goldin-Tolman in the context of
equivariant symplectic geometry \cite{GolTol09}).
This manuscript 
is another step in this direction, in that we develop a complete
Schubert-calculus-type description of the equivariant and ordinary
cohomology rings of a space which is not a $G/P$. 
Although our work and that of Goldin and Tolman are clearly related,
they are
different in nature; for instance, they always assume their spaces
are manifolds, while Peterson varieties are in general singular. 
Nevertheless, both our methods and those of
Goldin-Tolman depend heavily on GKM theory, to which we now turn.

GKM theory was presented by Goresky-Kottwitz-MacPherson \cite{GKM}
based on previous work of e.g. Chang-Skjelbred \cite{ChaSkj74} and
others. 
The original theory builds combinatorial tools to compute the
$T$-equivariant cohomology ring of a $T$-space $X$ that satisfies certain
technical conditions. This influential 
theory and its many consequences have been extensively generalized and
used since \cite{BCS06, GolTol09, GorKotMcP04, GH04, GZ01, GZ02, HHH05,
  KnuTao03, LauNgo04}.  In particular, extensions of GKM theory apply
to many of the generalized equivariant cohomology theories mentioned above. One of
the powerful features of GKM theory is that it allows us to build
convenient $H^*_T(\pt)$-module generators for the equivariant
cohomology $H^*_T(X)$ of the $T$-space $X$. (In an
equivariant-symplectic-geometric context, the elements of such a basis
can be given equivariant-Morse-theoretic interpretations in terms of
the moment map for the Hamiltonian $T$-action.) 
In the case of $\Gr(k, \C^n)$ or  
$\mathcal{F}\ell ags(\C^n)$, the equivariant Schubert classes give
precisely such a basis, 
thus allowing for
effective use of GKM theory in both classical and modern Schubert calculus
\cite{KnuTao03}. 

Unfortunately, classical GKM theory does \textit{not} apply
to our main objects of study, the type $A$ Peterson varieties.
Informally, this is because `the torus is too small'. More precisely,
we have the following. The Peterson variety $Y$ is a subvariety of
$\mathcal{F}\ell ags(\C^n)$. 
It is well-known that the torus action of $n \times n$ 
invertible diagonal matrices on
$\mathcal{F}\ell ags(\C^n)$ satisfies the technical conditions
required in GKM theory. However, 
the torus action of diagonal matrices does
\textit{not} preserve the Peterson variety $Y \subseteq
\mathcal{F}\ell ags(\C^n)$.  A circle subgroup of the torus does preserve $Y$, 
but this $S^1$-action on $Y$ does {\em not} satisfy
the GKM conditions. Nevertheless, we can explicitly analyze this
$S^1$-action and its fixed points $Y^{S^1}$, and obtain our first main
result (Theorem~\ref{theorem:pvA-basis}), which builds a
\textbf{computationally effective $H^*_{S^1}(\pt)$-module basis} for
$H^*_{S^1}(Y)$.  This basis satisfies certain crucial properties  in
GKM theory (also satisfied by the equivariant
Schubert classes in $H^*_T(\mathcal{F}\ell ags(\C^n))$), namely: 
\begin{enumerate}
\item upper-triangularity (see Equations~\eqref{eq:Peterson-upper-triangular}
  and~\eqref{eq:Peterson-upper-triangular-nonzero}) and 
\item minimality (see Equation~\eqref{eq:Peterson-minimality}).
\end{enumerate}
For precise statements and proof, see Theorem~\ref{theorem:pvA-basis}
and Proposition~\ref{prop:minimality}. In our situation, there is a natural map
$H^*_T(\mathcal{F}\ell ags(\C^n) \to H^*_{S^1}(Y)$ induced by inclusions of tori 
and varieties. Our module basis
of Theorem~\ref{theorem:pvA-basis} additionally satisfies the property
that
\begin{enumerate}
\item[(3)] each element of the basis is obtained as the image of an
  equivariant Schubert class in $H^*_T(\mathcal{F}\ell ags(\C^n))$. 
\end{enumerate}
Motivated by (3), we call our basis elements \textbf{Peterson Schubert
  classes}. They are indexed by subsets $\mathcal{A} \subseteq
\{1,2,\ldots, n-1\}$, and for the purposes of this section only,
we denote by $p_{\mathcal{A}}$ the Peterson Schubert class
corresponding to $\mathcal{A}$. It turns out that the previous three conditions characterize
our module basis $\{p_{\mathcal{A}}\}$ uniquely, in a suitable sense (Proposition~\ref{prop:pvA-unique}).

We now describe our second main result, the Chevalley-Monk formula for
Peterson varieties.  As a preliminary step, we first prove in
Proposition~\ref{prop:pi-generate} that the subset of
cohomology-degree-$2$ classes $p_i := p_{{\{i\}}}$ for \(1 \leq i \leq
n-1\) form a set of ring generators of $H^*_{S^1}(Y)$. Given this set
of ring generators, our \textbf{$S^1$-equivariant Chevalley-Monk
  formula formula for Peterson varieties} (see
Theorem~\ref{theorem:Monk}, where we use slightly different notation)
is a set of \textit{explicit} formulas to compute the product of an
arbitrary ring-generator class $p_i$ with an arbitrary
module-generator class $p_{\mathcal{A}}$.  We have
\begin{equation}\label{eq:Monk-intro}
p_i \cdot p_{\mathcal{A}} = c^{\mathcal{A}}_{i, \mathcal{A}} \cdot 
p_{\mathcal{A}} + \sum_{{\mathcal{A}} \subsetneq {\mathcal{B}} \textup{ and } |{\mathcal{B}}|
  = |{\mathcal{A}}|+1} c^{\mathcal{B}}_{i,{\mathcal{A}}} \cdot p_{{\mathcal{B}}}
\end{equation}
for any $i$ and $\mathcal{A}$, where the structure constants $c^{\mathcal{A}}_{i,\mathcal{A}}, c^{\mathcal{B}}_{i, \mathcal{A}}$ can be explicitly computed as follows.

First, 
\begin{itemize} 
\item $c^{\mathcal{A}}_{i, \mathcal{A}} = 0$ if $i \not \in
  \mathcal{A}$, 
\item $c^{\mathcal{A}}_{i, \mathcal{A}} = (\head{{\mathcal{A}}}{i} -
  i+1)(i-\toe{{\mathcal{A}}}{i} +1) t$ if $i \in \mathcal{A}$, 
\end{itemize} 
where the variable $t$ is the cohomology degree $2$ generator of $H^*_{S^1}(\pt) \cong \C[t]$. 
Additionally, for a subset $\mathcal{B} \subseteq \{1,2,\ldots,n-1\}$ which is a disjoint union
\(\mathcal{B} =
  \mathcal{A} \cup \{k\},\) we have explicit formulas, for which we need some notation. 
Given any set
$\mathcal{C} \subseteq \{1,2,\ldots, n-1\}$ and any \(k \in
\mathcal{C},\) denote by $\toe{{\mathcal{C}}}{k}$ and
$\head{{\mathcal{C}}}{k}$ the unique integers such that
\(\toe{{\mathcal{C}}}{k} \leq k \leq \head{{\mathcal{C}}}{k},\) the
consecutive sequence
\(\{\toe{{\mathcal{C}}}{k}, \toe{{\mathcal{C}}}{k}+1, \ldots,
\head{{\mathcal{C}}}{k}-1, \head{{\mathcal{C}}}{k}\}\) is a subset 
of 
$\mathcal{C}$, and such that \(\toe{{\mathcal{C}}}{k} -1 \not \in \mathcal{C},
\head{{\mathcal{C}}}{k}+1 \not \in \mathcal{C}.\) Then we have 
\begin{itemize} 
\item  $c^{\mathcal{B}}_{i,{\mathcal{A}}} =
0$ if \(i \not \in \{\toe{\mathcal{B}}{k}, \toe{\mathcal{B}}{k}+1,
  \ldots, \head{\mathcal{B}}{k}-1, 
\head{\mathcal{B}}{k}\},\)
\item if \(k \leq i \leq \head{\mathcal{B}}{k},\) then 
\[
c^{\mathcal{B}}_{i,{\mathcal{A}}} = (\head{{\mathcal{B}}}{k}-i+1) \cdot \left( \begin{array}{c} \head{{\mathcal{B}}}{k} - \toe{{\mathcal{B}}}{k}+1 \\
    k-\toe{{\mathcal{B}}}{k} \end{array} \right)
\]
\item if \(\toe{{\mathcal{B}}}{k} \leq i \leq k-1,\) then 
\[
c^{\mathcal{B}}_{i,{\mathcal{A}}} 
 = (i-\toe{{\mathcal{B}}}{k}+1) \cdot \binom{\head{{\mathcal{B}}}{k}-\toe{{\mathcal{B}}}{k}+1}{k-\toe{{\mathcal{B}}}{k}+1}.
\]
\end{itemize} 
An immediate consequence of the formulas above is 
that the (non-zero) structure constants $c^{\mathcal{B}}_{i, \mathcal{A}}$ are
both \textit{positive} and \textit{integral} in the appropriate
sense.
Moreover, our formula evidently has many of the desirable
properties advertised above: it is explicit, easily computed, 
and both \textbf{manifestly positive} and \textbf{manifestly
  integral}.

Finally, since the cohomology degree $2$ Peterson Schubert classes together with the 
pure equivariant class $t \in \C[t] \cong H^*_{S^1}(t)$ generate the ring 
$H^*_{S^1}(Y)$, our Chevalley-Monk
formula completely determines the $H^*_{S^1}(\pt)$-algebra structure
of the $S^1$-equivariant cohomology $H^*_{S^1}(Y)$. In particular, we
may explicitly describe $H^*_{S^1}(Y)$ as a ring with generators
$\{p_{\mathcal{A}}\}$ and $t$ satisfying precisely the relations~\eqref{eq:Monk-intro}, which we do in Corollary~\ref{corollary:ring-presentation-eqvt}. 
Moreover, it can be seen that the forgetful map \(H^*_{S^1}(Y)
\to H^*(Y)\) takes the Peterson Schubert classes to a $\C$-basis of 
the ordinary cohomology $H^*(Y)$, and the cohomology degree $2$
classes generate $H^*(Y)$ as a ring. Thus, as a straightforward
consequence of our $S^1$-equivariant Chevalley-Monk formula, we obtain
both a Chevalley-Monk formula for the ordinary cohomology $H^*(Y)$ of the
Peterson variety (Corollary~\ref{corollary:ordinary-Monk}), as well as an explicit generators-and-relations
description of $H^*(Y)$ (Corollary~\ref{corollary:ring-presentation-ordinary}). 
We expect these results to lead to a rich array of
further work.

The above discussion suggests the wide variety
of mathematics related to, and touching upon, this work. Indeed, our intended
audience consists of researchers interested in any subset of: 
Schubert calculus, combinatorics, equivariant algebraic topology,
geometric representation theory, algebraic geometry, or 
symplectic geometry.  For this reason we have attempted
to keep exposition elementary and prerequisites to a minimum. In
particular, we consistently use notation and terminology from type
$A$. Similarly, we favor specificity to generality throughout. 
An exception to this rule 
is the appendix, where we prove a general lemma in
Borel-equivariant cohomology with field coefficients, included here 
in this form to be of maximum use for our future work. 

We close with a discussion of avenues for further inquiry and a
sampling of open questions. 
First, we intend to explore the relationship
between our explicit presentation of the ordinary cohomology ring $H^*(Y)$ of
type $A$ Peterson varieties with conjectural
presentations due to A.~Mbirika. Mbirika's presentation is 
expressed in terms of `partial symmetric
functions' and Young tableaux, and directly generalizes  the
classical Borel presentation of $H^*(\mathcal{F}\ell ags(\C^n))$.
We already have preliminary results
which will be useful in this direction, including a Giambelli formula
for the equivariant cohomology of 
Peterson varieties. 
Second, and as mentioned above, we view our results here as the first successful example
of `generalized Schubert calculus' which extends
beyond the realm of Kac-Moody flag varieties $G/P$. 
In this manuscript, we heavily exploit the natural $S^1$-action on $Y$,
obtained by restricting an $(S^1)^n$-action on a larger GKM
space $X$ (in this case $\mathcal{F}\ell ags(\C^n)$). 
We intend to explore 
the more general case in which a $T'$-space $Y$ arises as a
$T'$-invariant subspace of a $T$-space $X$ which is GKM, for a subtorus $T'$ of
$T$. We have preliminary results which suggest that, under suitable
hypotheses, there exist appropriate `upper-triangular' module bases
for $H^*_{T'}(Y)$ 
similar to those constructed in this manuscript.
Finally, we conclude with several open questions which we hope to
address in future work.
\begin{itemize}
\item The structure constants
$c^{\mathcal{B}}_{i,\mathcal{A}}$ appearing in~\eqref{eq:Monk-intro} are non-negative integers.
Are the 
$c^{\mathcal{B}}_{i,\mathcal{A}}$ are some kind of intersection
numbers for suitable geometric objects corresponding to the
$p_{\mathcal{A}}$? 
\item In this manuscript, we restrict to Peterson varieties of Lie
  type $A$ and to Borel-equivariant cohomology with $\C$
  coefficients. 
Can our results can be generalized to
\begin{itemize}
\item general Lie type, 
\item general regular nilpotent Hessenberg varieties, and/or 
\item other generalized equivariant cohomology theories
  (e.g. equivariant $K$-theory)? 
\end{itemize}
\item Brion and Carrell have announced a result of Peterson's which gives
a presentation of the $S^1$-equivariant cohomology of the Peterson variety
\cite{BriCar04} which is different from ours.  What is the relationship between 
our presentation and theirs? 
\item Are there Springer-type representations on $S^1$-equivariant
  cohomology for all or some Peterson varieties? 
\end{itemize}

\bigskip

\noindent{\bf Acknowledgements.}  The authors are grateful to the
NSF-supported Midwest Topology Network for a generous travel
grant for research collaboration, which made some of this work
possible. The authors also acknowledge
the University of Iowa NSF VIGRE program, which supported
the `Equivariant geometry and combinatorics of flag
varieties' workshop at the University of Iowa in June 2009, during
which some of our ideas were generated. Both authors also thank Darius Bayegan for 
many helpful conversations.

\bigskip

\noindent{\bf Notation, terminology, and conventions.} 

\medskip

$n$ is a fixed but arbitrary positive integer. 

 $G$ denotes the Lie group $GL(n,\C)$.

 $B$ denotes the Borel subgroup of $G$ consisting of upper-triangular
 matrices. 

 $T$ is the compact maximal torus of the compact form $U(n,\C)$ of
 $G$, consisting of unitary diagonal matrices.

$\{t_i - t_{i+1}: 1 \leq
i \leq n-1\}$ is the set of positive simple roots of $\Lie(G)$.

$w \in S_n$ is expressed in one-line notation. Hence
\[
w = (w(1), w(2), \ldots, w(n)) \in S_n
\]
is the permutation on $n$ letters sending $i$ to $w(i)$. If $e_1, e_2, \ldots, e_n$ are the standard basis
vectors of $\mathbb{C}^n$, then the permutation matrix $w$ is related 
to the permutation $w \in S_n$
by $w e_i = e_{w(i)}$ for all $i$.

$s_i$ denotes the simple transposition in $S_n$ that interchanges $i$ and $i+1$
 and acts as the identity on all other elements of $\{1,2,\ldots, n\}$.

$s_i \cdot (t_j - t_{j+1})$, the action of the $s_i$ on the positive simple
roots $t_j - t_{j+1}$, is given by the action of $s_i$ on the indices of
the variables $t_k$. 

$w < w'$ in the \textbf{Bruhat order} if for any (hence every)
reduced-word decomposition of $w'$, there exists a subword which
equals $w$.  

$\ell(w)$ is the length of $w \in S_n$ with respect to the Bruhat
order, namely the minimal number $k$ of simple transpositions needed
to write $w = s_{i_1} s_{i_2} \cdots s_{i_k}$. 

$w_0$ is the \textbf{unique maximal element} of $S_n$; it has the property that 
it is Bruhat-larger than every other element of the group. 

$\mathbf{b} = (b_1, b_2, \ldots, b_{\ell(w)})$ denotes a reduced-word
decomposition of a permutation $w$. Here $\mathbf{b}$ is the sequence
of the indices of the simple transpositions whose product is $w$, so
\(w = s_{b_1} s_{b_2} \cdots s_{b_{\ell(w)}}.\)

$[a_1, a_2]$ for integers $a_1, a_2$ with $a_1 \leq a_2$ denotes the
set of consecutive integers $\{a_1, a_1+1, \ldots, a_2\}$.

\textbf{Equivariant cohomology}, in this manuscript, means
\textbf{Borel-equivariant cohomlogy with $\C$ coefficients}. 

\textbf{Restriction} refers to the natural map on (equivariant or ordinary)
cohomology induced by an inclusion map of spaces \(X_1 \into X_2.\) 
In the setting when $X_1$ is the set of fixed points of $X_2$ under a
group action, some manuscripts refer to this restriction map as a
\textbf{localization}; we avoid this terminology to prevent confusion
with other (e.g. Atiyah-Bott-Berline-Vergne) localization theories.

$\sigma_w$ is the $T$-equivariant Schubert class in
$H^*_T(G/B)$ corresponding to $w \in S_n$. We will abuse
notation and denote also by $\sigma_w$ the image of $\sigma_w$ under
the inclusion $H^*_T(G/B) \into H^*_T((G/B)^T)$.

$\Sym(\t^*) \cong \C[t_1, t_2, \ldots, t_n]$ is identified with the
$T$-equivariant cohomology $H^*_T(\pt)$.

$\Sym(\Lie(S^1)^*) \cong \C[t]$ is identified with the $S^1$-equivariant cohomology
$H^*_{S^1}(\pt)$. 

$Y$ denotes the Peterson variety in $G/B \cong \mathcal{F}\ell
ags(\C^n)$ of type $A_{n-1}$. 

$\mathcal{H}_{\mathcal{A}}$ and $\mathcal{T}_{\mathcal{A}}$ denote integer functions
as given in Definitions~\ref{defn:head} and~\ref{defn:tail}.

\bigskip

\section{Peterson varieties, $S^1$-actions, and
$S^1$-fixed points}\label{sec:fixed-points}

In Sections~\ref{subsec:def-Hess} and~\ref{subsec:torus-action} below,
we very briefly introduce the main characters of this manuscript
-- both the spaces and the torus (or circle) actions on
them. We refer the reader to \cite{Tym06} for a more leisurely
account.  Then in
Section~\ref{subsec:torus-fixed-pts}, we give an explicit
combinatorial enumeration of the $S^1$-fixed points of the Peterson
variety which will prove useful in the later sections. 

\subsection{Flag varieties, Hessenberg varieties, and Peterson
  varieties}\label{subsec:def-Hess}

The {\bf flag variety} (or {\bf flag manifold}) is the complex
homogeneous space $G/B$, which can also be described as the space of 
nested sequences of subspaces in $\C^n$.  Let 
\[
\mathcal{F}\ell ags(\C^n) := 
\{ V_{\bullet} = (V_1 \subseteq V_2 \subseteq \cdots V_{n-1} \subseteq
\C^n) \hsm \mid \hsm \dim_{\C}(V_i) = i \}. 
\]
The group $G$ acts naturally on $\mathcal{F}\ell ags(\C^n)$ by left
multiplication, namely $g \cdot V_{\bullet} := ( g \cdot
V_i)_{i=1}^n$. The stabilizer of a fixed flag $V_{\bullet}$ is
isomorphic to $B$; this provides the identification of $G/B$ with
$\mathcal{F}\ell ags(\C^n)$. 

\textit{Hessenberg varieties} (in type $A$) are subvarieties of
$\mathcal{F}\ell ags(\C^n) \cong G/B$, specified by pairs consisting
of an $n \times n$ complex matrix $X$ and a Hessenberg function $h$,
i.e. a nondecreasing function $h: \{1,2,\ldots,n\} \rightarrow
\{1,2,\ldots,n\}$.  Given such an $X$ and $h$, the Hessenberg variety
${\mathcal H}ess(X,h)$ is defined as 
\begin{equation}\label{eq:def-Hess}
\mathcal{H}ess(X,h) := \{ V_{\bullet}  \in \mathcal{F}\ell ags(\C^n) \;
\vert \;  X V_i \subseteq
V_{h(i)} \text{ for all } i=1,\ldots,n\} \subseteq \mathcal{F}\ell
ags(\C^n). 
\end{equation}
We say $\mathcal{H}ess(X,h)$ is a \textit{regular nilpotent
Hessenberg variety} if $X$ is a principal nilpotent operator, i.e. 
$X$ has a single Jordan block and its eigenvalue is zero.
More concretely, if $E_{i,j}$ denotes the $n \times n$ matrix whose
entries are zero except for a $1$ in the $(i,j)^{th}$ place, then up
to change of basis we may take
\begin{equation} \label{eqn:jordanform}
X = E_{1,2} + E_{2,3} + \cdots + E_{n-1,n}.
\end{equation}
If $X$ is a principal nilpotent operator and the Hessenberg
  function is given by $h(i) = i+1$ for $1 \leq i \leq n-1$ and
  $h(n)=n$ then $\mathcal{H}ess(X,h)$ is called a \textbf{Peterson
    variety} of Lie type $A_{n-1}$; we denote it by $Y$. 

For example, if $n=2$ then the Peterson variety is the full flag
variety.  If $n=3$ then the Peterson variety consists of the following flags:
\[\left\{ 
  V_1 = \left\langle \left( \begin{array}{c} 1 \\ 0 \\ 0 \end{array} \right) \right\rangle, 
  V_2 = \left\langle  \left( \begin{array}{c} 0 \\ 1 \\ 0 \end{array}
    \right) ,  \left( \begin{array}{c} 1 \\ 0 \\ 0 \end{array} \right)
  \right\rangle,
  V_3 = \C^3 \right\},\]
  
\[  \left\{ 
  V_1 = \left\langle \left( \begin{array}{c} a \\ 1 \\ 0 \end{array} \right) \right\rangle, 
  V_2 = \left\langle \left( \begin{array}{c} 1 \\ 0 \\ 0 \end{array}
    \right), \left( \begin{array}{c} a \\ 1 \\ 0 \end{array} \right) 
   \right\rangle,
  V_3 = \C^3 : \textup{ for all } a \in \C \right\},\]
  
\[  \left\{ 
  V_1 = \left\langle \left( \begin{array}{c} 1 \\ 0 \\ 0 \end{array} \right) \right\rangle, 
  V_2 = \left\langle \left( \begin{array}{c} 0 \\ b \\ 1 \end{array}
    \right), 
  \left( \begin{array}{c} 1 \\ 0 \\ 0 \end{array} \right) 
 \right\rangle,
  V_3 = \C^3  : \textup{ for all } b \in \C \right\}, \textup{ and }\]
  
 \[ \left\{ 
  V_1 = \left\langle \left( \begin{array}{c} c \\ d \\ 1 \end{array} \right) \right\rangle, 
  V_2 = \left\langle \left( \begin{array}{c} d \\ 1 \\ 0 \end{array}
    \right), 
   \left( \begin{array}{c} c \\ d \\ 1 \end{array} \right) 
 \right\rangle,
  V_3 = \C^3  : \textup{ for all } c,d \in \C \right\}.\]

\subsection{Torus actions on flag varieties and circle actions on Peterson varieties}\label{subsec:torus-action}

The flag variety $G/B \cong \mathcal{F}\ell ags(\C^n)$ is equipped
with a natural $T \cong (S^1)^n$-action coming from usual left multiplication of
cosets. This $T$-action has many useful properties: for
instance, there are finitely many
$T$-fixed points $wB \in G/B$, corresponding precisely to the permutation matrices $w \in S_n$.

However, this $T$-action does {\em not} restrict to the
Hessenberg varieties in $G/B$, in the sense that an arbitrary
Hessenberg variety is typically not preserved by the full $T$-action.
However, not all is lost: a natural $S^1$ subgroup of the maximal torus $T$
does preserve any Hessenberg variety $\mathcal{H}ess(X,h)$ whose matrix
$X$ is nilpotent and in Jordan canonical form. Consider the
$1$-dimensional subtorus 
\begin{equation}\label{eq:def-circle}
  \left\{ \left. \begin{bmatrix} t^n & 0 & \cdots & 0 \\ 0 & t^{n-1} &  &
      0 \\ 0 & 0 & \ddots & 0 \\ 0 & 0 &  & t \end{bmatrix}
  \; \right\rvert \;  t \in \C, \; \|t\| = 1 \right\}  \subseteq T^n
\subseteq U(n,\C) 
\end{equation}
of the maximal torus $T$, which we henceforth denote $S^1$.

The following are straightforward consequences of~\eqref{eq:def-Hess}
and~\eqref{eq:def-circle}; we leave proofs to the reader.

\begin{fact} 
The $S^1$
in~\eqref{eq:def-circle} preserves any Hessenberg variety
$\mathcal{H}ess(X,h)$ with $X$ nilpotent and in Jordan form. 
\end{fact} 

\begin{fact} 
The $S^1$-fixed points of $G/B$ are precisely the $T$-fixed
  points, i.e. $(G/B)^T  = (G/B)^{S^1}$. 
\end{fact} 

\begin{fact}
The $S^1$-fixed points in $\mathcal{H}ess(X,h)$ are the $T$-fixed
    points of $G/B$ that lie in $\mathcal{H}ess(X,h)$, namely
\begin{equation}\label{eq:Peterson-fixed}
\mathcal{H}ess(X,h)^{S^1} = (G/B)^T \cap \mathcal{H}ess(X,h).
\end{equation}
\end{fact}

\subsection{Combinatorial enumeration of  $S^1$-fixed points in the Peterson variety}\label{subsec:torus-fixed-pts}

It is straightforward from the definitions to check that the
$S^1$-fixed points~\eqref{eq:Peterson-fixed} for regular nilpotent Hessenberg varieties in type
$A_{n-1}$ are the permutations $w$ with $w^{-1}(i) \leq
h(w^{-1}(i+1))$ for all $i<n$.  In the case of Peterson varieties, 
this condition is equivalent to 
\begin{equation}\label{eq:fixedpt-condition}
w^{-1}(i) \leq w^{-1}(i+1)+1 \quad  \mbox{for all} \; 1 \leq i<n.
\end{equation}
In particular, either
$w^{-1}(i+1) > w^{-1}(i)$ or $w^{-1}(i+1) = w^{-1}(i)-1$.  
This means that the entries in the one-line notation for $w^{-1}$,
read from left to right, must either increase or, alternatively,
decrease by exactly $1$. The
one-line notation for $w^{-1}$ is therefore of the form
\begin{equation}\label{eq:w-oneline}
w^{-1} = ( j_1, j_1 -1, \ldots, 1, j_2, j_2-1, \ldots, j_1+1, \ldots,
n, n-1, \ldots, j_m+1),
\end{equation}
where \(1 \leq j_1 < j_2 < \cdots < j_m < n\) is any sequence of
strictly increasing integers. 
It turns out that for our purposes the complement in $\{1,2,\ldots,n-1\}$ of
the set $\{j_1, j_2, \ldots, j_m\}$ will be more useful.  Thus for each permutation \(w \in S_n\)
satisfying~\eqref{eq:fixedpt-condition} we define the subset of
$\{1,2,\ldots, n-1\}$ given by
\[
{\mathcal{A}} := \{ i: w^{-1}(i) = w^{-1}(i+1)+1 \textup{ for } 1 \leq i\leq n-1\} \subseteq
\{1,2,\ldots, n-1\}.
\]
Informally, $\mathcal{A}$ consists of those indices for which the
one-line notation of $w^{-1}$ decreases by $1$. This argument shows that the
permutations $w \in S_n$ satisfying~\eqref{eq:w-oneline} are in
bijective correspondence with the set of subsets $\mathcal{A}$. Furthermore, note that
the $n\times n$ permutation matrix associated to the $w^{-1}$ above is
block diagonal with blocks of size $j_1, (j_2 - j_1), \cdots,
(n-j_m)$, each of which has $1$'s on the antidiagonal and $0$
elsewhere. Thus a permutation $w^{-1}$ of
the form~\eqref{eq:w-oneline} is its own inverse: $w^{-1} = w$.
Henceforth we denote by $w_{\mathcal{A}} \in S_n$ the permutation
$w^{-1}=w =: w_{\mathcal{A}}$ corresponding as above to a subset
${\mathcal{A}} \subseteq \{1,2,\ldots,n-1\}$.

\begin{example}\label{example:S7}
Suppose $n = 7$ and ${\mathcal{A}} = \{1,2,3,5\}$. Then
\[
w_{\mathcal{A}} = (4,3,2,1,6,5,7)   \in S_7
\]
and the corresponding $7 \times 7$ permutation matrix is given by 
\[
\left[ \begin{array}{ccccccc} 0 & 0 & 0 & 1 & 0  & 0 & 0 \\ 0 & 0 & 1
    & 0 & 0 & 0 & 0 \\ 0 & 1 & 0 & 0 & 0 & 0 & 0 \\ 1 & 0 & 0 & 0 & 0
    & 0 & 0 \\ 0 & 0 & 0 & 0 & 0 & 1 & 0 \\ 0 & 0 & 0 & 0 & 1 & 0 & 0
    \\ 0 & 0 & 0 & 0 &  0 & 0 & 1 \end{array} \right].
\]
\end{example}

There is a natural decomposition of each set ${\mathcal{A}}$ into subsets
corresponding to the block submatrices in the permutation matrix
representation of $w_{\mathcal{A}}$. We make the following definition. 
\begin{definition}
A {\bf maximal consecutive (sub)string} of ${\mathcal{A}}$ is a set of consecutive
integers $\{a_1, a_1+1, \ldots, a_1+k\} \subseteq {\mathcal{A}}$ such that neither $a_1-1$
nor $a_1+k+1$ is in ${\mathcal{A}}$.  Let $a_2 := a_1+k$. We denote the corresponding
maximal consecutive substring by $[a_1,a_2]$.
\end{definition}
Any ${\mathcal{A}}$ uniquely decomposes into a disjoint union of
maximal consecutive substrings 
\[ {\mathcal{A}} = [a_1, a_2] \cup [a_3,a_4] \cup \cdots \cup [a_{m-1},a_m].\] 
In Example \ref{example:S7}, the maximal
consecutive strings are $\{1,2,3\}$ and $\{5\}$. 

Given a consecutive string $[a_j, a_{j+1}]=\{a_j, a_{j}+1,\ldots,
a_{j+1}-1,a_{j+1}\}$, the element $w_{[a_j, a_{j+1}]}$ is the largest element with respect to
Bruhat order in the subgroup $S_{[a_j, a_{j+1}]}$ of permutations of
$\{a_j, a_{j}+1,\ldots,a_{j+1}-1,a_{j+1}\}$. We fix the following
reduced-word decomposition of $w_{[a_j, a_{j+1}]}$:
\begin{equation}
  \label{eq:wjk-reduced-word}
  w_{[a_j, a_{j+1}]} = \prod_{k=0}^{a_{j+1}-a_j} \left( \prod_{i=0}^{a_{j+1}-a_j-k}
    s_{a_j+i} \right).
\end{equation}
Here we take the convention that a product is always composed from the
left to the right, so $\prod_{i=0}^{k} \beta_i = \beta_0 \cdot
\beta_1 \cdots \beta_k$ for any expressions $\beta_i$. 

\begin{example} 
Continuing with Example~\ref{example:S7}, for the maximal consecutive string \([1,3]:=\{1,2,3\}\)
we have 
\[
w_{[1,3]} = s_1 s_2 s_3 s_1 s_2 s_1.
\]
\end{example} 

A reduced-word decomposition for $w_{\mathcal{A}}$ is then obtained by taking
the product of the $w_{[a_j, a_{j+1}]}$ for each of the maximal consecutive
substrings $[a_j, a_{j+1}]$ of ${\mathcal{A}}$, ordered so that maximal consecutive substrings increase from left 
to right.  (Simple
transpositions commute if their indices differ by at least $2$, so if $[a_j, a_{j+1}], [a_j', a_{j+1}']$ are disjoint maximal 
consecutive substrings of ${\mathcal{A}}$ then
$w_{[a_j, a_{j+1}]}$ and $w_{[a_j', a_{j+1}']}$ commute.) 
In other words, suppose \({\mathcal{A}}
= [a_1, a_2] \cup [a_3,a_4] \cup \cdots \cup [a_{m-1},a_m]\) is a decomposition into 
maximal consecutive substrings of ${\mathcal{A}}$ with $a_1 < a_2
< \ldots < a_m$. Then we fix the reduced-word decomposition
\begin{equation}\label{eq:wA-reduced-word} 
  w_{\mathcal{A}} = w_{[a_1, a_2]} w_{[a_3, a_4]} w_{[a_5, a_6]}
  \cdots w_{[a_{m-1}, a_m]}. 
\end{equation}

\begin{example} 
Continuing further with Example~\ref{example:S7}, we have
\[w_{\mathcal{A}} = w_{[1,3]}w_{[5,5]} = s_1 s_2 s_3 s_1 s_2 s_1s_5.\]
\end{example}

\section{GKM theory on the flag variety and restriction to
  $S^1$-fixed points on Peterson varieties}\label{sec:GKM}

In this section, we describe the general framework used for
our computations. Our main conceptual tool is the well-known
GKM theory for $T$-spaces, as recounted in the introduction. 
Only two aspects of GKM theory are essential to our
discussion: first, we use the injectivity of the restriction map to
the equivariant cohomology of the torus-fixed points; and
second, we use certain special classes, which we call {\bf flow-up
  classes}, to build a natural module basis over the
equivariant cohomology of a point for the equivariant
cohomology of the $T$-space.

We begin by recalling well-known results. 
The flag variety \(G/B \cong \mathcal{F}\ell ags(\C^n)\) is equipped
with a natural $T$-action given by left multiplication on cosets; the fixed points are
precisely the isolated points $wB \in G/B$ corresponding to the permutations \(w \in
S_n.\)  The $T$-equivariant inclusion $\imath:~(G/B)^T
\into G/B$ induces a ring homomorphism from the $T$-equivariant cohomology of
$G/B$ to that of its $T$-fixed points, i.e. 
\begin{equation}\label{eq:GmodB-localization}
\imath^*: H^*_T(G/B) \into H^*_T((G/B)^T) \cong \bigoplus_{w \in S_n}
H^*_T(\pt)
\end{equation}
and it is well-known that $\imath^*$ is an injection.  Note that the
codomain of the restriction map~\eqref{eq:GmodB-localization} is a
direct sum of polynomial rings $H^*_T(\pt) \cong \Sym(\t^*)$. Since
$\imath^*$ is injective, we may therefore uniquely specify elements of
$H^*_T(G/B)$ as a list of polynomials in $\Sym(\t^*) \cong \C[t_1, t_2,
\ldots, t_n]$.

A classical result in Schubert calculus is that the $T$-equivariant
cohomology ring $H^*_T(G/B)$ has an $H^*_T(\pt)$-module basis given by
the ($T$-equivariant) Schubert classes $\{\sigma_w\}_{w \in
  S_n}$ \cite{BerGelGel73, Dem74}. By the above discussion, we may think of
$\sigma_w$ in terms of its image under $\iota^*$ in $H^*_T((G/B)^T)$,
which in turn we view as a function \(S_n \to \Sym(\t^*).\) Let
$\sigma_w(w') \in \Sym(\t^*)$ denote the value of $\sigma_w$ at $w'
\in S_n$.

The Schubert classes $\sigma_w$ satisfy certain computationally
convenient properties with respect to the Bruhat order on
$S_n$. First, they are \textbf{upper-triangular} in an appropriate
sense, namely: 
\begin{equation}\label{eq:GmodB-upper-triangular} 
\sigma_w(v) = 0 \quad \mbox{if } v \not \geq w
\end{equation}
and 
\begin{equation}\label{eq:GmodB-nondeg}
\sigma_w(w) \neq 0. 
\end{equation} 
Second, they are \textbf{minimal} among upper-triangular classes: if
$\sigma_{w'}$ satisfies the equations~\eqref{eq:GmodB-upper-triangular} for $w$,
then
\begin{equation}\label{eq:GmodB-minimality} 
\sigma_w(w) \textup{ divides } \sigma_{w'}(w).
\end{equation}

One of the main results of this manuscript is to construct a
suitable additive $H^*_{S^1}(\pt)$-module basis for the $S^1$-equivariant
cohomology of Peterson varieties, similar to the Schubert
classes in $H^*_T(G/B)$ in the sense that they satisfy analogous
upper-triangularity and minimality 
conditions. This allows us to develop a theory of
``generalized ($S^1$-equivariant) Schubert calculus'' in
the equivariant cohomology of Peterson varieties. Moreover, the module basis is obtained as a
subset of the images of the Schubert classes $\sigma_w$ in the
$S^1$-equivariant cohomology of the Peterson variety, as we explain in
Section~\ref{sec:module-basis}. Here and below, we set the stage for
this main result by developing the
necessary preliminary tools and terminology.

Let $Y$ denote the Peterson variety of type $A_{n-1}$. 
As seen in Section \ref{sec:fixed-points}, the variety $Y$ is naturally an
$S^1$-space for a certain subtorus $S^1$ of $T$; moreover 
\(Y^{S^1} = (G/B)^T \cap Y.\) Recall that there is a
natural forgetful map from $T$-equivariant cohomology to
$S^1$-equivariant cohomology obtained by the inclusion map of groups
\(S^1 \into T\).  These facts allow us to extend
the map~\eqref{eq:GmodB-localization} to the commutative diagram 
\begin{equation}\label{eq:comm-diag}
\xymatrix{
H^*_T(G/B) \ar[r] \ar[d] & H^*_T((G/B)^T) \ar[d] \\
H^*_{S^1}(G/B) \ar[r] \ar[d] & H^*_{S^1}((G/B)^T) \ar[d] \\
H^*_{S^1}(Y) \ar[r] & H^*_{S^1}(Y^{S^1}). 
}
\end{equation}

The images of the equivariant Schubert classes $\{\sigma_w\}$
under the composition of the natural maps $H^*_T(G/B) \to
H^*_{S^1}(G/B) \to H^*_{S^1}(Y)$ are 
crucial to our discussion, so we make a
definition. 

\begin{definition}\label{def:Peterson-Schubert}
  Let $\sigma_w$ be an equivariant Schubert class in $H^*_T(G/B)$. Let
  $p_w \in H^*_{S^1}(Y)$ be the image of $\sigma_w$ under the 
 ring map $H^*_T(G/B) \to H^*_{S^1}(Y)$ in~\eqref{eq:comm-diag}. We call $p_w$ the {\bf Peterson
    Schubert class corresponding to $w$}.
\end{definition}

We want to specify the Peterson Schubert
class $p_w$ by its image in
$H^*_{S^1}(Y^{S^1})$ via the bottom horizontal arrow
in~\eqref{eq:comm-diag}.  For this we need the following.

\begin{theorem}\label{theorem:injectivity}
Let $Y$ be the type $A_{n-1}$ Peterson variety, equipped with the
natural $S^1$-action defined by~\eqref{eq:def-circle}. Then 
\begin{itemize}
\item $H^*_{S^1}(Y) \cong H^*_{S^1}(\pt) \otimes H^*(Y)$ as
  $H^*_{S^1}(\pt)$-modules, and 
\item the inclusion $Y^{S^1} \into Y$ induces a ring map 
\[
\imath^*: H^*_{S^1}(Y) \to H^*_{S^1}(Y^{S^1})
\]
which is injective. 
\end{itemize}
\end{theorem}

\begin{proof} 
It is well-known (\cite[Chapter III, Section 14]{BottTu}, \cite[Chapter
6]{GS}) that if the 
ordinary cohomology of $Y$ is concentrated in even degree, then 
the Leray-Serre spectral sequence for the fibration \(Y \to Y \times_T
ET \to BT\) collapses, which then implies the first conclusion of the
theorem. 
Recall that the (complex) affine cells in a paving by affines\footnote{A paving by affines is like a cell decomposition, but the closure conditions on a paving by affines are weaker than for a 
cell decomposition.} of a complex algebraic variety induce homology
generators \cite[19.1.11]{Ful84}; in particular, since the cells are
complex, they are even-dimensional and hence a complex variety with a
paving by affines has ordinary cohomology only in even
degree. Peterson varieties in type $A_{n-1}$ admit a paving by affines
\cite{Tym06}. 
Moreover, the abstract localization theorem
\cite[Theorem 11.4.4]{GS} states that the kernel of $\imath^*$ is the
module of torsion elements in $H^*_T(Y)$. Since we have just seen that
$H^*_T(Y)$ is a free $H^*_T(\pt)$-module, the kernel must be $0$, and
$\imath^*$ is injective, as desired. 
\end{proof}

The theorem above implies that we may think of $p_w \in H^*_{S^1}(Y)$ purely in terms of their
images in $H^*_{S^1}(Y^{S^1})$, as in the case of equivariant Schubert
classes in $H^*_T(G/B)$. Since the restriction map is injective, we will abuse notation and
refer to the image of $p_w$ in $H^*_{S^1}(Y^{S^1})$ also as $p_w$. The $S^1$-fixed points of $Y$ are isolated so 
\[
H^*_{S^1}(Y^{S^1}) \cong \bigoplus_{w' \in Y^{S^1}} H^*_{S^1}(w') \cong
\bigoplus_{w' \in Y^{S^1}} \C[t].
\]
This means each $p_w$ is a function $Y^{S^1} \to \C[t]$ just as in the case of $G/B$. Following our
notation for $G/B$, if $w' \in Y^{S^1}
\subseteq (G/B)^T$ is a fixed point, we denote by $p_w(w')$ the value of the
restriction of $p_w$ to $w'$. 

Finally, we observe that the restrictions 
$p_w(w')$ may be computed using the restrictions $\sigma_w(w')$ of
the equivariant Schubert classes on $G/B$ and the maps in ~\eqref{eq:comm-diag} .

\begin{proposition}\label{prop:compute-pw}
  Let $Y$ be the type $A_{n-1}$ Peterson variety and let $p_w$ be a
  Peterson Schubert class corresponding to $w \in S_n$. Let $w' \in Y^{S^1} \subseteq (G/B)^T \cong
  S_n$. Then $p_w(w') \in H^*_{S^1}(\pt)$ is the image of $\sigma_w(w')
  \in \Sym(\t^*) \cong \C[t_1, t_2, \ldots, t_n]$
  under the projection map 
\begin{equation}\label{eq:def-piS1}
\xymatrix @R=.1in {
\pi_{S^1}: \C[t_1, t_2, \ldots, t_n] \ar[r] & \C[t]  \\
t_i \; \ar@{|->}[r] & (n-i+1) t. \\
}
\end{equation}
\end{proposition}

\begin{proof}
Recall that 
\[
H^*_T((G/B)^T) \cong \bigoplus_{w \in S_n} \Sym(\t^*) 
\cong \bigoplus_{w \in S_n} \C[t_1, t_2, \ldots, t_n]
\] 
and
\[
H^*_{S^1}((G/B)^T) \cong \bigoplus_{w \in S_n} \Sym(\Lie(S^1)^*) \cong
\bigoplus_{w \in S_n} \C[t].
\]
The top right
arrow in~\eqref{eq:comm-diag} that sends
$H^*_T((G/B)^T)  \to H^*_{S^1}((G/B)^T)$
is induced from the projection map \(\Sym(\t^*) \to
\Sym(\Lie(S^1)^*)\) coming from the inclusion \(\Lie(S^1) \into \t.\)
The definition of the subgroup $S^1$ in~\eqref{eq:def-circle} implies
that each $t_i$ projects to $(n-i+1)t$. For the bottom right
arrow in~\eqref{eq:comm-diag}, we recall that $(G/B)^{S^1} =
(G/B)^T$, as observed in Section~\ref{sec:fixed-points}. We then
see that the map
\[
\bigoplus_{w \in S_n} \C[t] \cong H^*_{S^1}((G/B)^{S^1}) \to
H^*_{S^1}(Y^{S^1}) \cong \bigoplus_{w \in Y^{S^1}} \C[t] 
\]
is the identity on each component corresponding to $w \in Y^{S^1}
\subseteq S_n \cong (G/B)^{S^1}$ and is $0$ on each component corresponding to $w \in
S_n \setminus Y^{S^1}$. More colloquially, it kills the components in
the direct sum 
associated to $S^1$-fixed points in $G/B$ which do 
not appear in $Y$. Composition of the two arrows and commutativity of the diagram in~\eqref{eq:comm-diag} give the desired
result. 
\end{proof}

\section{A $H^*_{S^1}(\pt;\Q)$-module basis for the $S^1$-equivariant
  cohomology of Peterson varieties}\label{sec:module-basis}

As recounted in Section~\ref{sec:GKM}, 
the equivariant Schubert classes $\{\sigma_w\}_{w \in S_n}$ have
properties which 
make them particularly convenient for Schubert-calculus
computations. One of the main results of this manuscript is an
explicit construction, in Theorem~\ref{theorem:pvA-basis}, of an
$H^*_{S^1}(\pt)$-module basis for the $S^1$-equivariant cohomology of
Peterson varieties which also satisfies upper-triangularity and
minimality conditions. As in classical Schubert calculus, this
makes the basis especially useful for explicit computations; we
exploit these properties to derive Monk formulas in
Section~\ref{sec:Monk}.

First we make precise the conditions satisfied by our module basis
of $H^*_{S^1}(Y)$. The upper-triangularity condition on
Schubert classes is stated in terms of the Bruhat order on permutations $w \in S_n$ viewed as
$T$-fixed points in $G/B$.  Bruhat order restricts to $Y^{S^1}$ since $Y^{S^1}$ is a subset of
$(G/B)^T \cong S_n$. 
We use this partial order, also called Bruhat order, on the
$S^1$-fixed points of $Y$.

Next we define permutations $v_{\mathcal{A}} \in S_n$ which are
naturally associated to each subset ${\mathcal{A}} \subseteq \{1,2,\ldots,
n-1\}$. We saw in Section~\ref{subsec:torus-fixed-pts} that $Y^{S^1}$ is enumerated by the set of subsets ${\mathcal{A}}$ of $\{1,2,\ldots, n-1\}$. We will see that the Peterson Schubert classes $p_{v_\mathcal{A}}$ associated
to the permutations $v_{\mathcal{A}}$ form an additive
$H^*_{S^1}(\pt)$-module basis for $H^*_{S^1}(Y)$, thus playing a role analogous to Schubert
classes in $H^*_T(G/B)$. We have the following.

\begin{definition}\label{def:vA}
  Let \({\mathcal{A}} = \{j_1 < j_2 < \cdots j_m\}\)
 be a subset of \(\{1,2,\ldots, n-1\}.\) We define the element $v_{\mathcal{A}} \in S_n$ 
 to be the product of simple transpositions whose indices
 are in ${\mathcal{A}}$, in increasing order, i.e. 
\[
v_{\mathcal{A}} := s_{j_1} s_{j_2} \cdots s_{j_m} = \prod_{i=1}^m s_{j_i}.
\]
\end{definition}

Each subset $\mathcal{A} \subseteq \{1,2,\ldots, n-1\}$
corresponds to a unique permutation of the form $v_{\mathcal{A}}$
so the collection of Peterson Schubert classes
$\{p_{v_{\mathcal{A}}}\}$ for all subsets ${\mathcal{A} \subseteq \{1,2,\ldots, n-1\}}$
gives rise to a collection of elements in $H^*_{S^1}(Y)$ in
one-to-one correspondence with the $S^1$-fixed points of $Y$.

Our next tasks are to show that this collection
$\{p_{v_{\mathcal{A}}}\}$ satisfies conditions analogous to~\eqref{eq:GmodB-upper-triangular}
with respect to the (restricted) Bruhat order. We enumerate the
conditions precisely. 
\begin{enumerate}
\item Upper-triangularity: 
\begin{equation}\label{eq:Peterson-upper-triangular} 
  p_{v_{\mathcal{A}}}(w_{\mathcal{B}}) = 0 \quad \mbox{if} \; \; w_{\mathcal{B}} \not \geq v_{\mathcal{A}} 
\end{equation}
and 
\begin{equation}\label{eq:Peterson-upper-triangular-nonzero}
 p_{v_{\mathcal{A}}}(w_{\mathcal{A}}) \neq 0. 
\end{equation}
\item Minimality: 
\begin{equation}\label{eq:Peterson-minimality}
p_{v_{\mathcal{A}}}(w_{\mathcal{A}}) \textup{ divides } p_w(w_{\mathcal{A}}) \textup{ in $\C[t]$}
\end{equation}
if  $p_w$ is any Peterson Schubert class satisfying the upper-triangularity condition \eqref{eq:Peterson-upper-triangular} for ${\mathcal{A}}$.
\end{enumerate}

We now prove that the $p_{v_{\mathcal{A}}}$ satisfy
the 
upper-triangularity condition, which will
naturally lead to our main Theorem~\ref{theorem:pvA-basis};
in the next section, we find that the collection
$\{p_{v_{\mathcal{A}}}\}$ satisfies the minimality condition and is unique in an appropriate sense
(Proposition~\ref{prop:pvA-unique}).

Note that the definition of $v_{\mathcal{A}}$ immediately implies that 
\[
s_j < v_{\mathcal{A}} \quad \textup{ for all } j \in {\mathcal{A}}. 
\]
We record some basic facts below which will be important in what
follows. The proofs are straightforward and left to the reader. 

\begin{fact}\label{fact:vA-length}
The Bruhat-length of $v_{\mathcal{A}}$ is the size of the set
$\mathcal{A}$, i.e. \(\ell(v_{\mathcal{A}}) = |\mathcal{A}|.\) In
particular, the decomposition in Definition~\ref{def:vA} is
minimal-length. 
\end{fact}

\begin{fact}\label{fact:vA-reduced-word-unique}
If $\mathcal{A} = [a_k, a_{k+1}]$ is a maximal consecutive string, then the word in
Definition \ref{def:vA} is the unique reduced word decomposition for
$v_{[a_k, a_{k+1}]}$. 
\end{fact} 

\begin{fact}\label{fact:vA-product}
If ${\mathcal{A}} =  [a_1, a_2] \cup [a_3,a_4] \cup \cdots \cup
[a_{m-1},a_m]$ is a decomposition 
of $\mathcal{A}$ into maximal consecutive substrings with $1 \leq a_1 < a_2 < \cdots < a_m < n$, then 
\[v_{\mathcal{A}} = v_{  [a_1, a_2] }v_{[a_3, a_4] } \cdots v_{  [a_{m-1}, a_m] }.\]
Moreover, there exists exactly one subword of $v_{\mathcal{A}}$ which
is equal to $v_{[a_i, a_{i+1}]}$. 
\end{fact} 

\begin{fact}\label{fact:unique-vA-in-wA}
  There exists exactly one reduced subword in the reduced word
  decomposition~\eqref{eq:wA-reduced-word} of $w_{\mathcal{A}}$ which
  is equal to $v_{\mathcal{A}}$.  (The proof uses
  uniqueness of the reduced word $v_{ [a_i, a_{i+1}] }$ for each
  minimal string, and examination of the definition of
  $w_{\mathcal{A}}$.)
\end{fact}

The next lemma is the crucial observation which allows us to show that
the Peterson Schubert classes $p_{v_{\mathcal{A}}}$ corresponding to these
special Weyl group elements $v_{\mathcal{A}}$ are a $H^*_{S^1}(\pt)$-module
basis for $H^*_{S^1}(Y)$. The essence is that the Bruhat order on
$Y^{S^1}$ can be translated to the usual partial order on sets given
by containment.

\begin{lemma}\label{lemma:containment}
Let ${\mathcal{A}},{\mathcal{B}}$ be subsets in $\{1,2,\ldots, n-1\}$.  Then $v_{\mathcal{A}} \leq w_{\mathcal{B}}$ if and only if ${\mathcal{A}} \subseteq {\mathcal{B}}$.
\end{lemma}

\begin{proof}
  If ${\mathcal{A}} \subseteq {\mathcal{B}}$ then $v_{\mathcal{A}}
  \leq w_{\mathcal{B}}$ by definition of $v_{\mathcal{A}}$ and
  $w_{\mathcal{B}}$.  Now suppose that $v_{\mathcal{A}} \leq
  w_{\mathcal{B}}$.  In particular this means that $s_i \leq
  v_{\mathcal{A}}$ for all $i \in {\mathcal{A}}$. Bruhat order is transitive
  so $s_i \leq w_{\mathcal B}$.  By definition of
  $w_{\mathcal{B}}$ this means $i \in {\mathcal{B}}$. Hence
  ${\mathcal{A}} \subseteq {\mathcal{B}}$ as desired. 
\end{proof}

We next develop 
tools to compute restrictions of
$p_{v_{\mathcal{A}}}$ at various fixed points $w_{\mathcal{B}} \in
Y^{S^1}$.  These methods allow us to prove the upper-triangularity
condition~\eqref{eq:Peterson-upper-triangular} with
respect to the partial order on sets (equivalent to the restriction of
Bruhat order by
Lemma~\ref{lemma:containment}). 
We begin with terminology.

\begin{definition}
Given a permutation $w$, a choice
of reduced-word decomposition ${\bf b}=(b_1, b_2, \ldots, b_{\ell(w)})$ of
$w$, and an index $i \in \{1,2,\ldots,\ell(w)\}$, define 
\begin{equation}\label{eq:def-beta}
r(i, \mathbf{b}) := s_{b_1} s_{b_2} \cdots s_{b_i-1} (t_{b_i} - t_{b_i+1}).
\end{equation}
\end{definition}

By definition $r(i, \mathbf{b})$ is an element
of $\Sym(\t^*) \cong \C[t_1, t_2, \ldots, t_n]$ of the form $t_j - t_k$ for some $j,k$. 
Classical results also show that $r(i,\mathbf{b})$ is in
fact a \textit{positive} root, namely, it has the form $t_j - t_k$
for $j<k$ \cite[Equation 4.1 and discussion]{Bil99}.  These positive
roots $r(i, \mathbf{b})$ are the building blocks of Billey's formula.

\begin{theorem}\label{theorem:billey} (``Billey's formula'', \cite[Theorem 4]{Bil99})
Let $w \in S_n$.  Fix a reduced word decomposition $w = s_{b_1}
s_{b_2} \cdots s_{b_{\ell(w)}}$ and let $\mathbf{b}= (b_1, b_2,
\ldots, b_{\ell(w)})$ be the sequence of its indices. 
Let $v \in S_n$. Then the value of the Schubert class $\sigma_v$ at
the $T$-fixed point $w$ is given by 
\begin{equation}\label{eq:billey-formula}
\sigma_v(w) = \sum r(i_1, \mathbf{b})r(i_2, \mathbf{b}) \cdots
r(i_{\ell(v)}, \mathbf{b})
\end{equation}
where the sum is taken over subwords $s_{b_{i_1}} s_{b_{i_2}} \cdots s_{b_{i_{\ell(v)}}}$ of $\mathbf{b}$ that are reduced words for $v$.
\end{theorem}

We refer to an individual summand of the
expression~\eqref{eq:billey-formula}, corresponding to a single reduced
subword $v=s_{b_{i_1}} s_{b_{i_2}} \cdots s_{b_{i_{\ell(v)}}}$ of $w$,
 as a {\bf
  summand in Billey's formula}. 
  
The following is a well-known consequence of the preceding discussion and theorem.

\begin{fact}\label{fact:billey-positive}
Each summand in Billey's formula for $\sigma_v(w)$ is a degree $\ell(v)$ polynomial in
the simple roots $\{t_i - t_{i+1}\}_{i=1}^{n-1}$ with {\em nonnegative} integer coefficients. 
\end{fact}

This is because each $r(i,\mathbf{b})$ is
a \textit{positive} root, namely a non-negative integral linear combination
of simple positive roots. 
Fact~\ref{fact:billey-positive} is sometimes summarized by saying Billey's formula is
\textit{positive in the sense of Graham} \cite{Gra01}.
This positivity implies that if any
summand in Billey's formula for $\sigma_v(w)$ is nonzero, then the
entire sum is nonzero. From this we derive the following. 

\begin{corollary}\label{cor:pvAwA-nonzero} 
Let ${\mathcal{A}} \subseteq \{1, 2, \ldots, n-1\}$. Then 
\[
p_{v_{\mathcal{A}}}(w_{\mathcal{A}}) \neq 0.
\]
\end{corollary}

\begin{proof}
We noted in Fact~\ref{fact:unique-vA-in-wA}
that $v_{\mathcal{A}}$ can be found as a subword of $w_{\mathcal{A}}$.
This implies that
$\sigma_{v_{\mathcal{A}}}(w_{\mathcal{A}}) \neq 0$ by the positivity
(in the sense of Graham) 
of Billey's formula. 
The projection $\Sym(\t^*) \cong \C[t_1, t_2, \ldots, t_n]
\to \C[t]$ sends each $t_i - t_{i+1}$ to
$t$.  Hence the image in $\C[t]$ of any nonzero polynomial in the
$t_i-t_{i+1}$ with positive coefficients is also nonzero in $\C[t]$. 
\end{proof}

The proof of the corollary above also shows the following.

\begin{proposition}\label{prop:deg-pvA}
Let $\mathcal{A}, \mathcal{B}$ be subsets of $\{1,2,\ldots, n-1\}$. 
Then 
\begin{itemize} 
\item the restriction $p_{v_{\mathcal{A}}}(w_{\mathcal{B}})$ of the Peterson
Schubert class $p_{v_{\mathcal{A}}}$ at any $w_{\mathcal{B}}$ has 
degree $\ell(v_{\mathcal{A}}) = \abs{\mathcal{A}}$ as a polynomial in
$\C[t]$, 
\item $p_{v_{\mathcal{A}}}$ has cohomology degree $2 \abs{\mathcal{A}}$, and
\item $p_{v_{\mathcal{A}}}(w_{\mathcal{B}})$ is non-zero if
$\sigma_{v_{\mathcal{A}}}(w_{\mathcal{B}})$ is non-zero. 
\end{itemize} 
\end{proposition}

We may now prove our first main theorem.

\begin{theorem}\label{theorem:pvA-basis}
Let $Y$ be the Peterson variety of type $A_{n-1}$ with the
natural $S^1$-action defined by~\eqref{eq:def-circle}. For ${\mathcal{A}}
\subseteq \{1,2,\ldots, n-1\}$, let $v_{\mathcal{A}} \in S_n$ be the permutation
given in Definition~\ref{def:vA}, and let $p_{v_{\mathcal{A}}}$ be the
corresponding Peterson Schubert class in $H^*_{S^1}(Y)$. 
The classes $\{p_{v_{\mathcal{A}}}: {\mathcal{A}} \subseteq \{1,2,\ldots, n-1\} \}$ in
$H^*_{S^1}(Y)$ 
\begin{itemize}
\item form an $H^*_{S^1}(\pt)$-module basis for $H^*_{S^1}(Y)$, and 
\item satisfy the upper-triangularity conditions 
\begin{equation}\label{eq:pvA-at-wB}
p_{v_{\mathcal{A}}}(w_{\mathcal{B}}) = 0 \quad \mbox{if } {\mathcal{B}} \not \supseteq {\mathcal{A}},  
\end{equation}
and
\begin{equation}\label{eq:pvA-at-wA-nonzero}
p_{v_{\mathcal{A}}}(w_{\mathcal{A}}) \neq 0.
\end{equation}
\end{itemize}
\end{theorem}

\begin{proof}
  We begin with a proof of the upper-triangularity
  condition~\eqref{eq:pvA-at-wB}. 
  Recall that $v_{\mathcal{A}} \leq w_{\mathcal{B}}$ if and only if
  ${\mathcal{A}} \subseteq {\mathcal{B}}$ by
  Lemma~\ref{lemma:containment}.  The image of zero
  under the map $\pi_{S^1}$ of Proposition~\ref{prop:compute-pw} is still zero, so
  it suffices to show that $\sigma_{v_{\mathcal{A}}}(w_{\mathcal{B}})
  = 0$ if $v_{\mathcal{A}} \not \leq w_{\mathcal{B}}$. This follows from the upper-triangularity of
  equivariant Schubert classes~\eqref{eq:GmodB-upper-triangular} (or can be proven
  directly from Billey's formula). 

  The assertion that $p_{v_{\mathcal{A}}}(w_{\mathcal{A}}) \neq 0$ is
  the content of Corollary~\ref{cor:pvAwA-nonzero}. 

  We now show that assertions~\eqref{eq:pvA-at-wB}
  and~\eqref{eq:pvA-at-wA-nonzero} imply that the
  $\{p_{v_{\mathcal{A}}}\}$, ranging over subsets ${\mathcal{A}}$ of
  $\{1,2,\ldots, n-1\}$, are $H^*_{S^1}(\pt)$-linearly independent.  Suppose $\sum_{\mathcal{A}} c_{\mathcal{A}}
  p_{v_{\mathcal{A}}} = 0 \in H^*_{S^1}(Y)$ for \(c_{\mathcal{A}} \in
  H^*_{S^1}(\pt).\) If any subset ${\mathcal{A}}$ has
  $c_{\mathcal{A}} \neq 0,$ then there must exist a minimal such, say
  ${\mathcal{B}}$.  Evaluating at $w_{\mathcal{B}}$,
  we conclude that 
\begin{equation}\label{eq:localization-at-wB}
\sum_{\mathcal{A}} c_{\mathcal{A}} \cdot p_{v_{\mathcal{A}}}(w_{\mathcal{B}}) =0.
\end{equation}
  By hypothesis on
  ${\mathcal{B}},$ we have $c_{\mathcal{A}} = 0$ for all ${\mathcal{A}}
  \subsetneq {\mathcal{B}}$. On the other hand, by~\eqref{eq:pvA-at-wB} we know
  $p_{v_{\mathcal{A}}}(w_{\mathcal{B}})=0$ for all ${\mathcal{A}} \not
  \subseteq {\mathcal{B}}$.  Hence the equality~\eqref{eq:localization-at-wB} simplifies to
\[
c_{\mathcal{B}} \cdot p_{v_{\mathcal{B}}}(w_{\mathcal{B}})=0.
\]
 From~\eqref{eq:pvA-at-wA-nonzero} and the fact that $H^*_{S^1}(\pt)
 \cong \C[t]$ is an integral domain,  we conclude $c_{\mathcal{B}} = 0$, a
  contradiction. Hence the $\{p_{v_{\mathcal{A}}}\}$ are linearly
  independent over $H^*_{S^1}(\pt)$. 

  Facts~\ref{fact:vA-length} and~\ref{fact:billey-positive} show that for any \(w \in Y^{S^1}\) the degree of
  the polynomial $p_{v_{\mathcal{A}}}(w)$ is
  $\abs{\mathcal{A}}$. The polynomial variable $t$ has
  cohomology degree $2$ so the cohomology degree of
  $p_{v_{\mathcal{A}}}$ in $H^*_{S^1}(Y)$ is $2\abs{\mathcal{A}}$.
  Since the $p_{v_{\mathcal{A}}}$ are enumerated precisely by the
  subsets $\mathcal{A}$ of $\{1,2,\ldots, n-1\}$, 
  we 
  conclude that there are $\binom{n-1}{j}$ distinct Peterson Schubert
  classes $p_{v_{\mathcal{A}}}$ of cohomology degree precisely $2j$. 
A result of Sommers and Tymoczko \cite{SomTym06} states that
$\binom{n-1}{j}$  is precisely the $2j$-th Betti number of $H^*(Y)$. 
Hence by Proposition~\ref{prop:general-module-gens} the $\{p_{v_{\mathcal{A}}}\}$ form an $H^*_{S^1}(\pt)$-module
basis for $H^*_{S^1}(Y)$, as desired. 
\end{proof}

\section{Combinatorial formulas for restrictions of Peterson Schubert
  classes to $S^1$-fixed points}\label{sec:combo}

In this section we explicitly evaluate the restrictions of a Peterson Schubert class
$p_{v_{\mathcal{A}}}$ at certain 
$S^1$-fixed points $w_{\mathcal{B}}$.  This will give a closed-form expression for the values 
$p_{v_{\mathcal{A}}}(w_{\mathcal{B}})$ needed in Section~\ref{sec:Monk}.
We also use these results to show that the module basis $\{p_{v_{\mathcal{A}}}\}$ satisfies the minimality
condition~\eqref{eq:Peterson-minimality}, and is the unique (in an
appropriate sense) set of Peterson Schubert classes in $H^*_{S^1}(Y)$ satisfying both the upper-triangularity and minimality conditions.

Our formulas will arise from a careful analysis of Billey's
formula, although our main interest is not in the
$\sigma_{v_{\mathcal{A}}}(w_{\mathcal{B}})$ but rather their
images $p_{v_{\mathcal{A}}}(w_{\mathcal{B}})$ via the projection
map $\pi_{S^1}$ in~\eqref{eq:def-piS1}.  This motivates us to establish the following terminology.

\begin{definition}
Let $v,w \in S_n$ and let $\sigma_v(w) = \sum r(i_1, \mathbf{b})r(i_2, \mathbf{b}) \cdots
r(i_{\ell(v)}, \mathbf{b})$ be Billey's formula for the restriction.  Using the projection $\pi_{S^1}$ of Proposition \ref{prop:compute-pw}, we refer to the expression
\[p_v(w) = \sum \pi_{S^1}(r(i_1, \mathbf{b})) \pi_{S^1}(r(i_2, \mathbf{b})) \cdots
\pi_{S^1}(r(i_{\ell(v)}, \mathbf{b}))\]
as {\bf Billey's formula for} $p_w$.
\end{definition}

We will proceed by first
explicitly computing the projection to $H^*_{S^1}(Y)$ of each of
the factors $r(i, \mathbf{b})$ in each of the summands of Billey's formula for $\sigma_w$. From this,
we derive concrete, explicit expressions for the terms in Billey's formula for $p_{v_{\mathcal{A}}}(w_{\mathcal{B}})$.

We begin with the
special case when ${\mathcal{A}}$ consists of a single maximal
consecutive string.
Before stating the lemma, we recall that the positive roots of
$G=GL(n,\C)$ have the form 
\[
t_j - t_{k+1}
\]
for \(j<k+1,\) and each such root may be expressed as a sum of
positive simple roots as follows: 
\[
t_j - t_{k+1} = (t_j - t_{j+1}) + (t_{j+1} - t_{j+2}) + \cdots + (t_k-t_{k+1}).
\]
The {\bf length} of the positive root $t_j-t_{k+1}$
is $k-j+1$.  Recall that Proposition \ref{prop:compute-pw} showed that 
\[\pi_{S^1}(t_j - t_{k+1}) = (k+1-j)t.\]

\begin{lemma}\label{lemma:length-of-roots} 
Let ${\mathcal{A}} =[a_1,a_2] \subseteq
  \{1,2,\ldots, n-1\}$ consist of a single maximal consecutive string,
  let $w_{\mathcal{A}}$ be the
  corresponding Weyl group element, and let
 $\mathbf{b} = (b_1, \ldots, b_{\ell(w_{\mathcal{A}})})$ be the reduced
 word decomposition of $w_{\mathcal{A}}$ given
  in~\eqref{eq:wA-reduced-word}. Fix an index $b_m$ for some $1 \leq m
  \leq \ell(w_{\mathcal{A}})$. Then 
\begin{enumerate} 
\item the index $b_m$ equals $i$ for some $a_1 \leq i \leq a_2$, 
\item $r(m, \mathbf{b})$ is a positive root of length $i-a_1+1$, and 
\item \(\pi_{S^1}(r(m, \mathbf{b})) = (i- a_1+1) t. \) 
\end{enumerate} 
In particular, the projection $\pi_{S^1}(r(m, \mathbf{b}))$  and 
the length of the positive root $r(m, \mathbf{b})$
depend only on the value of the index
$b_m$ and not on its position $m$ in $\mathbf{b}$. 

\end{lemma}

\begin{proof}
  The first claim is immediate from the fact that $\mathcal{A} = [a_1, a_2]$ and the definition of $w_{\mathcal{A}}$.  We prove the latter two claims by induction on the length of the consecutive string. The base case is when ${\mathcal{A}} =[a] = \{a\}$ is a singleton set, $\ell(w_{\mathcal{A}})=1$, and $w_{\mathcal{A}} = s_{a}$ is a single simple transposition. In this case the only possible choice of $b_m$ is $m=1$ and $b_m = a$. Moreover $r(m=1, \{a\})$ is
 $t_a-t_{a+1}$, which is a positive root of length $1$. By Proposition \ref{prop:compute-pw}, the root $t_a - t_{a+1}$ maps to $t$, so 
\[
\pi_{S^1}(r(1, \{a\})) = (a-a+1)t = 1 \cdot t 
\]
as desired. 

Now suppose that the consecutive string is $[a_1,a_2]$ with $a_2>a_1$ and that the
 lemma holds for the
consecutive string $[a_1,a_2-1]$. By definition of $w_{[a_1,a_2]}$
and $w_{[a_1,a_2-1]}$ in~\eqref{eq:wjk-reduced-word}, we have 
\begin{equation}\label{eq:wjk-prod}
w_{[a_1,a_2]} = s_{a_1} s_{a_1+1} \cdots s_{a_2} w_{[a_1,a_2-1]}.
\end{equation}
Let $\mathbf{b}_{[a_1,a_2-1]}$ and $\mathbf{b}_{[a_1,a_2]}$ be the reduced word decompositions of $w_{[a_1, a_2-1]}$ and $w_{[a_1,a_2]},$ respectively, given by~\eqref{eq:wjk-reduced-word}. Then 
\begin{equation}\label{eq:Ijk-union}
\mathbf{b}_{[a_1,a_2]} = \{b_1 = a_1, b_2 = a_1+1, \cdots, b_{a_2-a_1+1}=a_2\} \cup
\mathbf{b}_{[a_1,a_2-1]}
\end{equation}
where the union is of \emph{ordered} sequences.
We now prove the lemma holds for the first $a_2-a_1+1$ indices in $\mathbf{b}_{[a_1,a_2]}$,
i.e. for $b_m$ when $1 \leq m \leq a_2-a_1+1$. 
Direct calculation shows that for any
such $b_m$ we have 
\[
r(m,  \mathbf{b}_{[a_1,a_2]}) = s_{a_1} s_{a_1+1} \cdots s_{b_m-1} (t_{b_m} - t_{b_m+1}) = t_{a_1} - t_{b_m+1} 
\]
which by definition of $\pi_{S^1}$ projects to 
\[
\pi_{S^1}(r(m,  \mathbf{b}_{[a_1,a_2]})) = (b_m - a_1+1) t. 
\]
This proves the result for the first $a_2-a_1+1$ indices appearing in
$\mathbf{b}_{[a_1,a_2]}$. 

We now prove the result for indices $b_m \in  \mathbf{b}_{[a_1,a_2]}$ with
$m > a_2-a_1+1$. By observations~\eqref{eq:wjk-prod}
and~\eqref{eq:Ijk-union}, the $m$-th element $b_m = i$ in $\mathbf{b}_{[a_1,a_2]}$ for
$m > a_2-a_1+1$ is the $(m-(a_2-a_1+1))$-th element in $\mathbf{b}_{[a_1,a_2-1]}$ and
\[r(m, \mathbf{b}_{[a_1,a_2]}) = s_{a_1} s_{a_1+1} \cdots s_{a_2} r({m-(a_2-a_1+1)}, \mathbf{b}_{[a_1,a_2-1]}).\] 
By the inductive assumption
$r({m-(a_2-a_1+1)}, \mathbf{b}_{[a_1, a_2-1]})$ is a positive root $t_j - t_{j+i-a_1+1}$ of length
$i-a_1+1$. Note that $i$, $j$, and $j+i-a_1$ are all in $[a_1, a_2-1]$
by definition of $w_{[a_1,a_2-1]}$. A computation
shows 
\[
r(m,  \mathbf{b}_{[a_1,a_2]}) = s_{a_1} s_{a_1+1} \cdots s_{a_2} (t_j - t_{j+i-a_1+1}) = t_{j+1} - t_{j+i-a_1+2}.
\]
So $r(m,  \mathbf{b}_{[a_1,a_2]})$ is also a positive root of
length $i-a_1+1$, and this length depends only on the value $i$ of the
index $b_m$ and not on its location in $\mathbf{b}_{[a_1,a_2]}$. By definition
of $\pi_{S^1}$ we see $\pi_{S^1}(r(m,\mathbf{b}_{[a_1,a_2]})) =
(i-a_1+1)t$. 
\end{proof}

\begin{example}

Let $w = w_{[1,3]} = s_1 s_2 s_3 s_1 s_2 s_1$.  Then we have

 \[ \begin{array}{c|c|c|c|c|c|c}
 j & 1 & 2 & 3 & 4 & 5 & 6 \\
 \cline{1-7} \vspace{0.25em} r(j,\mathbf{b}_{[1,3]}) & t_1-t_2 & t_1 - t_3 & t_1 - t_4 & t_2 - t_3 &
 t_2 - t_4 & t_3 - t_4 \\
 \cline{1-7} \pi_{S^1}(r(j,\mathbf{b}_{[1,3]})) & t & 2t & 3t & t & 2t & t \end{array}
\]

\end{example}

The above lemma says the maximal consecutive substring containing $i
\in {\mathcal{A}}$ determines the corresponding factor in each summand
of Billey's formula.  This motivates the following definitions.

\begin{definition}\label{defn:head}
Fix ${\mathcal{A}} \subseteq \{1,2,\ldots,n-1\}$.  Let
$\mathcal{H}_{\mathcal{A}}: \mathcal{A} 
\to \mathcal{A}$ be the function such that 
\[\head{{\mathcal{A}}}{j} = \textup{the maximal element (``the head'') in the maximal consecutive substring of ${\mathcal{A}}$ containing $j$}.\]
\end{definition}

\begin{definition}\label{defn:tail}
Fix ${\mathcal{A}} \subseteq \{1,2,\ldots,n-1\}$.  Let $\mathcal{T}_{\mathcal{A}}: \mathcal{A}
\rightarrow \mathcal{A}$ be the function such that
\[\toe{{\mathcal{A}}}{j} = 
\textup{the minimal element (``the tail'') in the maximal consecutive substring of ${\mathcal{A}}$ containing } j.\]
\end{definition}

\begin{example}
For example if ${\mathcal{A}} = \{1,2,3,5,6\}$ then
\[\begin{tabular}{c|c|c|c|c|c}
j & 1 & 2 & 3  & 5 & 6 \\
\cline{1-6} $\toe{{\mathcal{A}}}{j}$ & 1 & 1 & 1 &  5 & 5\\
\cline{1-6} $\head{{\mathcal{A}}}{j}$ & 3 & 3 & 3 &  6 & 6\\
\end{tabular}\]
\end{example}

Using these functions, we may describe the
$p_{v_{\mathcal{A}}}(w_{\mathcal{B}})$ concretely. Building on the
previous lemma, we obtain the following expression for the
summands in Billey's formula for $p_{v_{\mathcal{A}}}$. 

\begin{lemma}\label{lemma:summands-equal}
  Let ${\mathcal{A}}, {\mathcal{B}} \subseteq \{1, \ldots,
  n-1\}$. If $\mathcal{A} \subseteq \mathcal{B}$ then each
  summand in Billey's formula for
  $p_{v_{\mathcal{A}}}(w_{\mathcal{B}})$ is 
\begin{equation}\label{eq:pvAwB-summand-formula}
\left( \prod_{j \in {\mathcal{A}}} (j-\toe{{\mathcal{B}}}{j}+1)\right) t^{|{\mathcal{A}}|}.
\end{equation}
In particular, all summands in
Billey's formula for $p_{v_{\mathcal{A}}}(w_{\mathcal{B}})$ are equal. 
\end{lemma}

\begin{proof}

  Let
  $\mathbf{b} = (b_1, b_2, \ldots, b_{\ell(w_{\mathcal{B}})})$ be the reduced-word decomposition for
  $w_{\mathcal{B}}$ given in~\eqref{eq:wA-reduced-word}. Let $J =
  \{i_1, i_2, \ldots,i_{|J|}\} \subseteq [1,\ell(w_{\mathcal{B}})]$ be
  a choice of subset of indices of $\mathbf{b}$ so that $s_{b_{i_1}}
  s_{b_{i_2}} \cdots s_{b_{i_{|J|}}} = v_{\mathcal{A}}$.
Then there is an equality of sets $\{
  b_{i_1}, b_{i_2}, \ldots, b_{i_{|J|}} \} = {\mathcal{A}}$ by
  Facts~\ref{fact:vA-reduced-word-unique} and~\ref{fact:vA-product}. 
The image under $\pi_{S^1}$ of the
  summand in Billey's formula for 
  $\sigma_{v_{\mathcal{A}}}(w_{\mathcal{B}})$ corresponding to the
  subword specified by $J$ is
  a product of terms $\pi_{S^1}(r(j, \mathbf{b}))$ for $j \in J$.
  Lemma~\ref{lemma:length-of-roots} implies that this product is
\[
\prod_{j \in {\mathcal{A}}} \left( (j - \toe{{\mathcal{B}}}{j} + 1) t \right) =
\left( \prod_{j \in {\mathcal{A}}} (j - \toe{{\mathcal{B}}}{j}+1) \right) t^{|{\mathcal{A}}|}
\]
as desired. 
\end{proof}

\begin{example}\label{ex:Billey-summand}
Suppose $n \geq 6, \mathcal{A}=\{1,2,3,5,6\}$ and $\mathcal{B}= \{1,2,3,4,5,6\}$.  
\begin{itemize}
\item Each summand in Billey's formula for $p_{v_{\mathcal{A}}}(w_{\mathcal{A}})$ is $(3!) \cdot (2!) t^5 = 12t^5$.  
\item Each summand in Billey's formula for $p_{v_{\mathcal{A}}}(w_{\mathcal{B}})$ is $(3!) \cdot (6 \cdot 5) t^5 = 180t^5$.  
\item Each summand in Billey's formula for $p_{v_{\mathcal{B}}}(w_{\mathcal{B}})$ is $(6!) t^6$.
\end{itemize}
\end{example}

Since all the summands in Billey's formula for $p_{v_{\mathcal{A}}}(w_{\mathcal{B}})$
are equal, we conclude the following. 

\begin{proposition}\label{prop:number-of-ways}
Let ${\mathcal{A}} \subseteq {\mathcal{B}} \subseteq \{1,2,\ldots,
n-1\}$. Let $p_{v_{\mathcal{A}}}$ be the Peterson Schubert class
corresponding to $v_{\mathcal{A}}$ and $w_{\mathcal{B}}$ the
permutation corresponding to $\mathcal{B}$. Then 
\begin{equation}\label{eq:number-of-ways}
p_{v_{\mathcal{A}}}(w_{\mathcal{B}}) = \left( \text{number of distinct subwords of $w_{\mathcal{B}}$ equal to $v_{\mathcal{A}}$} \right) \left( \prod_{j \in
    {\mathcal{A}}} (j - \toe{{\mathcal{B}}}{j}+1) \right) t^{|{\mathcal{A}}|}.
\end{equation}
\end{proposition}

\begin{example}\label{ex:Billey-sum} 
Continuing the previous example, suppose $n \geq 7, \mathcal{A}=\{1,2,3,5,6\}$ and $\mathcal{B}= \{1,2,3,4,5,6\}$.  The reader can check that 
\begin{itemize}
\item $p_{v_{\mathcal{A}}}(w_{\mathcal{A}}) = (3!) \cdot (2!) t^5 = 12t^5$.  
\item $p_{v_{\mathcal{A}}}(w_{\mathcal{B}}) = \binom{6}{3} \cdot (3!) \cdot (6 \cdot 5) t^5 = 3600t^5$.  
\item $p_{v_{\mathcal{B}}}(w_{\mathcal{B}}) = (6!) t^6$.
\end{itemize}
\end{example}

\begin{remark}
  In Section~\ref{sec:Monk}, we give explicit formulas for
  counting the number of ways to find $v_{\mathcal{A}}$ in $w_{\mathcal{B}}$ for
  special cases of ${\mathcal{B}}$ and ${\mathcal{A}}$ relevant for the equivariant Chevalley-Monk
  formula. 
\end{remark}

We can now give an explicit combinatorial formula for the value
of $p_{v_{\mathcal{A}}}$ at the fixed point $w_{\mathcal{A}}$.

\begin{corollary}\label{cor:pvAwA}
\begin{equation}
p_{v_{\mathcal{A}}}(w_{\mathcal{A}}) = \prod_{i \in {\mathcal{A}}} (i-\toe{{\mathcal{A}}}{i}+1) t^{\ell(v_{\mathcal{A}})}. 
\end{equation}
\end{corollary} 

\begin{proof}
  We observed in Fact~\ref{fact:unique-vA-in-wA} that exactly one
  subword of $w_{\mathcal{A}}$ is a reduced word decomposition of
  $v_{\mathcal{A}}$. 
 The desired result is now a corollary of the previous proposition.
\end{proof}

Next we show that the Peterson Schubert classes $\{p_{v_{\mathcal{A}}}\}$ satisfy the minimality condition~\eqref{eq:Peterson-minimality}.

\begin{proposition}\label{prop:minimality}
Let 
 $p_w$ be a Peterson Schubert class for $w \in S_n$ and suppose that
 $p_w$ satisfies 
\[
p_w(w_{\mathcal{B}}) = 0 \quad \mbox{for all} \; \; w_{\mathcal{B}} \not \geq w_{\mathcal{A}},
\]
and 
\[
p_w(w_{\mathcal{A}}) \neq 0.
\]
Then 
\[
p_{v_{\mathcal{A}}}(w_{\mathcal{A}}) \hspace{0.5em} \vert \hspace{0.5em} p_w(w_{\mathcal{A}}) \quad \mbox{in} \quad H^*_{S^1}(\pt).
\]
\end{proposition}

\begin{proof}
Let $w \in S_n$ as above. 
We first claim that $s_i \leq w$ if and only if $i \in
\mathcal{A}$. To see this, observe that if $p_w(w_{\mathcal{A}}) \neq
0$, then by Billey's formula, any reduced word for $w_{\mathcal{A}}$
contains a subword which equals $w$. In particular, if $s_i \leq w$
then $s_i$ also appears in every reduced word for
$w_{\mathcal{A}}$. Thus \(i \in A.\) 

To show the converse, we argue by contradiction. Suppose there exists
\(i \in \mathcal{A}\) with \(s_i \not \leq w.\) By the above argument,
this implies that there exists a proper subset ${\mathcal C}
\subsetneq {\mathcal{A}}$ such that $w$ is generated by $\{s_i: i \in
{\mathcal C} \subsetneq {\mathcal{A}}\}$. Denote this subgroup by $S_{\mathcal C}$.
Since $w_{\mathcal C}$ is by definition the longest word in
$S_{\mathcal C}$ and $w \in S_{\mathcal C}$, it follows that $w \leq
w_{\mathcal C}$.  Billey's formula implies $p_w(w_{\mathcal C}) \neq
0$, but $w_{\mathcal C} \not \geq w_{\mathcal{A}}$ since ${\mathcal C}
\subsetneq {\mathcal{A}}$, contradicting the upper-triangularity
assumption on $p_w$. Hence if \(i \in \mathcal{A}\) then
\(s_i \leq w\).

Now let $\mathbf{b}$ be the reduced word decomposition for
$w_{\mathcal{A}}$ given in~\eqref{eq:wA-reduced-word}.  Lemma
\ref{lemma:length-of-roots} states that the projection $\pi_{S^1}
(r(j, \mathbf{b}))$ of each factor of Billey's formula depends only
on the root $b_j$ and not on the location $j$ in $\mathbf{b}$.  Since
$s_i \leq w$ for each $i \in {\mathcal{A}}$, we conclude that the
product
\[
\prod_{i \in \mathcal{A}} \left( \left( i - \toe{\mathcal{A}}{i}
    +1\right) t\right) 
\]
divides each summand in Billey's
formula for $p_w(w_{\mathcal{A}})$. On the other hand, Corollary~\ref{cor:pvAwA} shows
that $p_{v_{\mathcal{A}}}(w_{\mathcal{A}}) =
\prod_{i \in {\mathcal{A}}} ((i - \toe{{\mathcal{A}}}{i}+1)t)$.  Hence
each summand in Billey's formula for $p_w(v_{\mathcal{A}})$ is
divisible by 
$p_{v_{\mathcal{A}}}(w_{\mathcal{A}})$.
Since each summand is divisible by
  $p_{v_{\mathcal{A}}}(w_{\mathcal{A}})$, so is the sum $p_w(v_{\mathcal{A}})$. 
\end{proof}

Finally, we prove that the classes $\{p_{v_{\mathcal{A}}}\}$ are
uniquely specified among all Peterson Schubert classes by their
upper-triangularity properties and their values at the appropriate
$w_{\mathcal{A}}$.  We emphasize that the uniqueness statement given
below in
Proposition~\ref{prop:pvA-unique} is at the level of cohomology
classes in $H^*_{S^1}(Y)$ and not at the level of elements $w \in
S_n$. More specifically, since the projection $H^*_T(G/B) \to
H^*_{S^1}(Y)$ is \textit{not} one-to-one, there may exist multiple $w
\in S_n$ such that $p_w = p_{v_A}$. This latter subtlety is explored
further in Proposition~\ref{prop:preimages-pvA}.

\begin{proposition}\label{prop:pvA-unique}
Let ${\mathcal{A}} \subseteq \{1,2,\ldots, n-1\}$. Suppose \(w \in
S_n\) is a permutation such that the corresponding Peterson
Schubert class $p_w$ satisfies the upper-triangularity condition for
${\mathcal{A}}$, i.e. 
  \[
p_w(w_{\mathcal{B}}) = 0 \quad \mbox{for all} \; \; w_{\mathcal{B}} \not \geq w_{\mathcal{A}},
\]
and agrees with $p_{v_{\mathcal{A}}}$ at $w_{\mathcal{A}}$, i.e. 
\[
p_w(w_{\mathcal{A}}) = p_{v_{\mathcal{A}}}(w_{\mathcal{A}}). 
\]
Then $p_w = p_{v_{\mathcal{A}}}$.
\end{proposition}

\begin{proof}
Any Peterson Schubert class $p_w$ is a homogeneous-degree class in
cohomology. The restriction of $p_w$ at
$w_{\mathcal{A}}$ agrees with that of $p_{v_{\mathcal{A}}}$.  By
Proposition~\ref{prop:number-of-ways} the class $p_{v_{\mathcal{A}}}$ 
has cohomology degree $2\abs{\mathcal{A}}$. 
Hence both $p_w$ and $p_w - p_{v_{\mathcal{A}}}$ have
cohomology degree $2\abs{\mathcal{A}}$.

Theorem~\ref{theorem:pvA-basis} says the
\(\{p_{v_{\mathcal{A}}}\}\) form a $H^*_{S^1}(\pt)$-basis for
$H^*_{S^1}(Y)$, so there are $c_{\mathcal{B}} \in
H^*_{S^1}(\pt)$ with
\[
p_w - p_{v_{\mathcal{A}}} = \sum_{\mathcal{B}} c_{\mathcal{B}} \cdot 
p_{v_{\mathcal{B}}}.
\]
Suppose that some $c_{\mathcal{B}} \neq 0$. Let $\mathcal{A}'$
be a minimal set with $c_{\mathcal{A}'} \neq 0$, meaning there is no $\mathcal{B}$
with $\mathcal{B} \subsetneq \mathcal{A}'$ with $c_{\mathcal{B}} \neq
0$. The upper-triangularity properties of the
$p_{v_{\mathcal{B}}}$ imply
\[
(p_w - p_{v_{\mathcal{A}}})(w_{\mathcal A'}) = c_{\mathcal{A}'} \cdot p_{v_{\mathcal{A}'}}(w_{\mathcal{A}'}). 
\]
By assumption on $\mathcal{A}'$, Corollary~\ref{cor:pvAwA}, and the
fact that $H^*_{S^1}(\pt)$ is an integral domain, the
right hand side of the above equality must be nonzero. Hence the left
hand side must also be non-zero. By the upper-triangularity conditions
on $p_w - p_{v_{\mathcal{A}}}$ and since $p_w(w_{\mathcal{A}}) = p_{v_{\mathcal{A}}}(w_{\mathcal{A}})$, we conclude that $\mathcal{A}
\subsetneq \mathcal{A}'$. In particular $2\abs{\mathcal{A}'} >
2\abs{\mathcal{A}}$ and consequently the cohomology degree of
$p_{v_{\mathcal{A'}}}$ is strictly greater than the cohomology degree
of $p_w - p_{v_{\mathcal{A}}}$. Moreover, any $H^*_{S^1}(\pt)$-multiple of
 $p_{v_{\mathcal{A}'}}$ must also be of cohomology degree strictly
greater than $p_w - p_{v_{\mathcal{A}}}$. Hence we achieve a
contradiction if any $c_{\mathcal{A}'} \neq 0$. We conclude all
coefficients are zero and that $p_w - p_{v_{\mathcal{A}}} = 0$, as
was to be shown. 
\end{proof}

\begin{remark} \label{remark:A-determines-pvA}
This proposition implies we may use the notation $p_{\mathcal A}$ instead of $p_{v_{\mathcal{A}}}$ to denote without ambiguity the Peterson Schubert class in $H^*_{S^1}(Y)$ that satisfies the upper-triangularity and minimality conditions for $w_{\mathcal A}$.  To maintain consistency, we will not change notation in this paper.
\end{remark}

As discussed above, Proposition~\ref{prop:pvA-unique} does not imply
uniquess at the level of permutations in $S_n$. Indeed, it is not
difficult to verify that if \(\mathcal{A} = [a_1, a_2]\) is a single
consecutive string, then 
\[
p_{(v_{\mathcal{A}})^{-1}} = p_{v_{\mathcal{A}}}
\]
as cohomology classes in $H^*_{S^1}(Y)$, although for most choices of
such $\mathcal{A}$ the two permutations are different.  We now prove that
this is essentially the only other permutation $w$ with $p_w =
p_{v_{\mathcal{A}}}$. 
 
 \begin{proposition}\label{prop:preimages-pvA}
Let $\mathcal{A} = [a_1,a_2] \subseteq \{1,2,\ldots,n-1\}$ be a maximal consecutive string with at least two elements.  Then $(v_{\mathcal{A}})^{-1}$ is the only permutation $w \neq v_{\mathcal{A}}$ with $p_w = p_{v_{\mathcal{A}}}$.
 \end{proposition}
 
 \begin{proof}
   Suppose $w \neq v_{\mathcal{A}}$ and $p_w = p_{v_{\mathcal{A}}}$.
   Then $w < w_{\mathcal{A}}$ since $p_w(w_{\mathcal{A}})$ is nonzero;
   in particular $w > s_i$ only if $i \in \mathcal{A}$.  On the other
   hand, for any \(\mathcal{B} \subsetneq \mathcal{A},\) we must have
   \(w \not
   < w_{\mathcal{B}}\) since $p_w(w_{\mathcal{B}})$ is zero for all
   $\mathcal{B} \subsetneq \mathcal{A}$ by assumption; in particular for all \(i \in
   \mathcal{A}\) the simple transposition $s_i$ must appear in a
   reduced word decomposition of $w$, i.e. $s_i < w$ if $i
   \in \mathcal{A}$.  Since $p_w(w_{\mathcal{A}}) =
   p_{v_{\mathcal{A}}}(w_{\mathcal{A}})$ we conclude that $\ell(w) =
   \ell(v_{\mathcal{A}})$.  Hence $w$ is a permutation of the simple
   transpositions $s_i$ for all $ i \in \mathcal{A}$.  The Peterson
   Schubert class $p_w$ corresponding to any such $w$ satisfies the upper-triangularity condition for
   $\mathcal{A}$ so it suffices to find $w$ that satisfy the
   minimality condition.  By Proposition \ref{prop:number-of-ways},
   this is equivalent to finding $w$ that appear exactly once as a
   subword of $w_{\mathcal{A}}$.
 
 We induct on the size $|{\mathcal{A}}|$ of $\mathcal{A}$.  Let $\mathcal{A} = \{a_1, a_1+1\}$.  
There are exactly two words of length two in the letters $s_{a_1}, s_{a_1+1}$.  By direct calculation $p_{s_{a_1}s_{a_1+1}}(w_{\mathcal{A}}) = p_{s_{a_1+1}s_{a_1}}(w_{\mathcal{A}})$.  Hence the claim holds if $|\mathcal{A}| = 2$.
 
Now suppose the claim holds when $|{\mathcal{A}}| = j-1$ and let $|{\mathcal{A}}| = j$.  Exactly one of $s_{a_2} s_{a_2-1}$ and $s_{a_2-1}s_{a_2}$ is a subword of $w$.  The simple transposition $s_j$ commutes with $s_{a_2}$ if $j \in \{a_1, a_1+1,\ldots,a_2-2\}$.  Hence either $w = s_{a_2}w'$ or $w = w's_{a_2}$ depending on the relative position of $s_{a_2-1}$ and $s_{a_2}$.  We treat each case separately.
Recall also that the simple reflection $s_{a_2}$ appears exactly once in $w_{\mathcal{A}}$ and that $w_{\mathcal{A}} = s_{a_1} s_{a_1+1} \cdots s_{a_2} w_{[a_1, a_2 - 1]}$.   

Suppose $w = s_{a_2}w'$.  If \(w' \neq v_{[a_1, a_2-1]}\) or if \(w' \neq (v_{[a_1, a_2-1]})^{-1},\) then by the inductive hypothesis there are at least two subwords of $w_{[a_1,a_2-1]}$ that equal $w'$, which in turn implies there are at least two subwords of $w_{\mathcal{A}}$ equal to $w$. This contradicts the assumption on $w$, so either $w' = v_{[a_1,a_2-1]}$ or $w'=(v_{[a_1,a_2-1]})^{-1}$.  Now suppose $w' = v_{[a_1,a_2-1]}$. Then there are at least two subwords of $w_{\mathcal{A}}$ that equal $w$, namely the subword corresponding to $s_{a_1} s_{a_1+1} \cdots s_{a_2-2}s_{a_2} s_{a_2-1}$ and the subword corresponding to $s_{a_2} s_{a_1} s_{a_1+1} \cdots s_{a_2-2}s_{a_2-1}$, which again contradicts the hypothesis on $w$.  Finally suppose $w' = (v_{[a_1, a_2-1]})^{-1}$. Then $w = (v_{[a_1, a_2]})^{-1}$, and a direct calculation shows that $w_{\mathcal{A}}$ has a unique subword that equals $(v_{[a_1,a_2-1]})^{-1}$.  Hence \(w = (v_{[a_1, a_2]})^{-1}\) is the only word of the form $s_{a_2}w'$ that satisfies our hypotheses. 

Now suppose $w = w' s_{a_2}$.  If $w' \neq v_{[a_1,a_2-1]}$ then by
definition of $v_{[a_1, a_2-1]}$ and assumption on $w'$, the
indices of the simple transpositions in a reduced word decomposition
of $w'$ are not strictly increasing. In particular there exists an
index $j$ such that $s_{j+1} s_{j+2} \cdots s_{a_2}$ is a subword of
$w$ and $s_j s_{j+1} s_{j+2} \cdots s_{a_1}$ is not a subword of $w$.
Since each of $s_{a_1}, s_{a_1+1}, \ldots, s_{j-1}$ commutes with any
of the $s_{j+1}, s_{j+2}, \cdots, s_{a_2}$ and since $s_j$ commutes with all of $s_{j+2}, s_{j+3}, \cdots, s_{a_2}$ we may write $w$ as 
\[w = s_{j+1} s_{j+2} \cdots s_{a_2} w''\]
where $w''$ is a permutation of the transpositions $s_{a_1},
s_{a_1+1}, \ldots, s_j$.  The assumption on $w$ implies $j \neq
a_2-1$, so there is at least one way to insert $s_{j+1}, \ldots,
s_{a_2-1}$ into $w''$ so that it is neither $v_{[a_1,a_2-1]}$ nor
$(v_{[a_1,a_2-1]})^{-1}$.  Applying the inductive hypothesis to this
word, we conclude that $w''$ is a subword of $w_{[a_1,a_2-1]}$ in at
least two ways.  
This in turn implies that $w$ occurs as a subword of $w_{\mathcal{A}}$
in at least two ways, contradicting the assumption on $w$. 
Hence $w = v_{[a_1, a_2-1]} s_{a_1} = v_{[a_1,a_2]}$ is the only word
of the form $w' s_{a_2}$ that satisfies our hypotheses, completing the
proof. 
 \end{proof}

 If $\mathcal{A}$ has $k$ maximal consecutive substrings of size at
 least two, Lemma \ref{disjoint product} below shows that there are $2^k$
 different Peterson Schubert classes $p_w$ with $p_w =
 p_{v_{\mathcal{A}}}$.  These Peterson Schubert classes correspond to
 all possible choices of either $v_{[a_i,a_{i+1}]}$ or
 $(v_{[a_i,a_{i+1}]})^{-1}$ on each maximal substring.

\section{A manifestly-positive equivariant Monk formula for Peterson varieties}\label{sec:Monk}

One of the central problems of modern Schubert calculus is to find 
concrete combinatorial formulas for the (ordinary or
equivariant) structure constants in the (ordinary or equivariant,
generalized) cohomology rings, with respect to the special module basis of
Schubert classes. 
In line with this general philosophy, we
therefore ask for concrete combinatorial methods
to compute products $p_{v_{\mathcal{A}}} \cdot
p_{v_{\mathcal{B}}}$ of Peterson Schubert classes
$\{p_{v_{\mathcal{A}}}\}$, which we showed in Section~\ref{sec:module-basis} form an 
$H^*_{S^1}(\pt)$-module basis for $H^*_{S^1}(Y)$. 

In this section, we partly achieve this goal: we
prove an \textbf{$S^1$-equivariant Chevalley-Monk formula} (also called a
\textbf{Monk formula}) in the
$S^1$-equivariant cohomology of the Peterson variety, i.e. 
we obtain an explicit, combinatorial formula for the
product of an \textit{arbitrary Peterson Schubert class} with a
\textit{Peterson Schubert class of cohomology degree $2$}. 
As a word of caution, we note that the terminology in the literature is ambiguous. 
For instance, in the Schubert calculus of the classical
  Grassmanian, the term ``Chevalley-Monk formula'' 
  refers to a formula for the product of an arbitrary Schubert class
  with an arbitrary cohomology degree $2$ class (the `single-box'
  class), while a ``Pieri formula" refers to a formula for the
  product of an arbitrary Schubert class with an arbitrary ``special''
  Schubert class (the `single-row' classes), which generate the cohomology ring but 
  may have cohomology
  degree $\geq 2$. In other cases, the use of terminology seems to
  depend on the relative importance ascribed by the authors to the two possible
  definitions of the subset of `special classes': either `degree $2$' or
  `generate cohomology ring'.  This results in ambiguity in cases when
  the two definitions agree. For instance, in the case of the flag variety, ``Chevalley'' is
  sometimes used to refer to formulas for products with `single-box'
  classes \cite{Wil07}, sometimes ``Pieri'' or ``Pieri-Chevalley'' 
  refers to formulas for products with `single-box' classes
  \cite{PitRam99}, and sometimes ``Pieri'' is used for formulas with
  `single-row' classes \cite{LenSot07, Rob02}. 
    We adhere to the
  \textit{Iowa convention}, a standardization of terminology negotiated at a small Schubert calculus workshop
    in 2009 at the University of Iowa: we
  refer to formulas for multiplication by cohomology-degree-$2$
  classes as \textit{Chevalley-Monk} (or \textit{Monk}) formulas, while we refer to
  formulas for multiplication by ``special classes'' of degree $\geq
  2$ as \textit{Pieri} formulas.

We also prove that our Monk formula completely
determines the $S^1$-equivariant cohomology $H^*_{S^1}(Y)$ of the
Peterson variety, namely that the cohomology-degree-$2$
classes generate $H^*_{S^1}(Y)$ as a ring. Moreover, we show that our
Monk formula is quite simple in that ``most terms are zero'' (made
precise below), and that the structure 
constants in our Monk formula are \textbf{non-negative} and
\textbf{integral}, either literally or in the sense of Graham,
depending on the polynomial degree of the structure constant. This
yields an explicit description via generators and relations of
$H^*_{S^1}(Y)$. Finally, we give analogues of the above results in the
context of the \emph{ordinary} cohomology $H^*(Y)$ of the Peterson
variety.

We begin with a definition for notational convenience.

\begin{definition}
Let $p_i$ denote the class $p_{s_i} \in H^*_{S^1}(Y)$, i.e. the Peterson Schubert class $p_{v_{\mathcal{A}}}$ for the singleton ${\mathcal{A}} = \{i\}$.
\end{definition}

From Proposition~\ref{prop:deg-pvA}, the set of
$\{p_i\}_{i=1}^{n-1}$ are exactly the cohomology degree $2$ classes
among the Peterson Schubert classes. We now prove that these, together with
one more degree $2$ class coming from $H^*_{S^1}(\pt)$, are in
fact \textit{ring} generators for $H^*_{S^1}(Y)$. Recall that
the $H^*_{S^1}(\pt)$-module structure of $H^*_{S^1}(Y)$ comes from the
ring map \(\pi_{BS^1}^*: H^*_{S^1}(\pt) \to H^*_{S^1}(Y)\) induced
from the projection $\pi_{BS^1}$ in the fiber bundle \(Y \to Y
\times_{S^1} ES^1 \stackrel{\pi_{BS^1}}{\longrightarrow} BS^1.\) In particular we view the
equivariant element \(t \in \C[t] \cong
H^*_{S^1}(\pt)\) of cohomology degree $2$ also as an element of $H^*_{S^1}(Y)$. 
We have the following.

\begin{proposition}\label{prop:pi-generate}
Let $Y$ be the type $A_{n-1}$ Peterson variety, equipped with the
natural $S^1$-action defined by~\eqref{eq:def-circle}. 
The Peterson Schubert classes $\{p_i: i = 1,
\ldots, n-1\}$ of cohomology degree $2$ together with the pure
equivariant degree $2$ class $t \in H^*_{S^1}(Y)$ generate the $S^1$-equivariant cohomology
$H^*_{S^1}(Y)$ as a ring.
\end{proposition}

\begin{proof}
It is well-known that $H^*_T(G/B)$ is generated in degree $2$, as is
$H^*_{S^1}(G/B)$. Since the restriction map $H^*_{S^1}(G/B) \to
H^*_{S^1}(Y)$ is surjective, the same holds true for
$H^*_{S^1}(Y)$. We have already seen that the 
$\{p_{v_{\mathcal{A}}}\}_{\mathcal{A} \subseteq \{1,2,\ldots,n-1\}}$
are a $H^*_{S^1}(\pt)$-module basis, and in particular the 
subspace of $H^*_{S^1}(Y)$ of degree $2$ is $\C$-spanned by 
\(\{p_i\}_{i=1}^{n-1}\) and the single `pure equivariant' class $t$. 
The result follows. 
\end{proof}

Monk's
formula is an explicit relationship between ring generators 
and module generators. 
More precisely, the fact that the set $\{p_{v_{\mathcal{A}}}\}$ form a
module basis for $H^*_{S^1}(Y)$ implies that for any $p_i$ and
$p_{v_{\mathcal{A}}}$ there exist structure constants
$c^{{\mathcal{B}}}_{i,{\mathcal{A}}} \in H^*_{S^1}(\pt) \cong \C[t]$ 
such that 
\begin{equation}\label{eq:general-Monk} 
p_i \cdot p_{v_{\mathcal{A}}} = \sum_{\mathcal{B}} c^{{\mathcal{B}}}_{i,{\mathcal{A}}} \cdot p_{v_{\mathcal{B}}}.
\end{equation}
Our main theorem of this section
provides a simple 
combinatorial formula for the $c^{\mathcal{B}}_{i,
  \mathcal{A}}$. Its proof has
several steps which occupy the rest of this section. 

We begin by proving that a simple condition on the
subsets $\mathcal{B}$ guarantees that the corresponding
structure constants $c^{\mathcal{B}}_{i,\mathcal{A}}$ are zero. This
allows us to refine the summation on the right hand side
of~\eqref{eq:general-Monk} and to obtain some simple formulas for
structure constants, as below.

\begin{proposition}\label{proposition:Monk-refined-1}
Let $\mathcal{A} \subseteq \{1,2,\ldots, n-1\}$ and $i \in
\{1,2,\ldots, n-1\}$. Then 
\begin{equation}\label{eq:Monk-refined}
p_i \cdot p_{v_{\mathcal{A}}} = c^{{\mathcal{A}}}_{i,{\mathcal{A}}} p_{v_{\mathcal{A}}} + \sum_{{\mathcal{A}} \subsetneq {\mathcal{B}} \textup{ and } |{\mathcal{B}}|
  = |{\mathcal{A}}|+1} c^{\mathcal{B}}_{i,{\mathcal{A}}} \cdot p_{v_{\mathcal{B}}},
\end{equation}
where 
\begin{enumerate} 
\item  $c^{\mathcal{A}}_{i,{\mathcal{A}}} = p_i(w_{\mathcal{A}})$ and 
\item if \(\mathcal{A} \subsetneq \mathcal{B}\) and
  \(\abs{\mathcal{B}} = \abs{\mathcal{A}} +1,\) then 
\begin{equation}\label{eq:cBiA-formula}
c^{\mathcal{B}}_{i,{\mathcal{A}}} =   (p_i(w_{\mathcal{B}}) -
p_i(w_{\mathcal{A}}))
\frac{p_{v_{\mathcal{A}}}(w_{\mathcal{B}})}{p_{v_{\mathcal{B}}}(w_{\mathcal{B}})}.
\end{equation}
\end{enumerate} 
\end{proposition}

\begin{proof} 
 For simplicity, in this argument we use
  the polynomial degree of the Peterson Schubert classes instead of
  the cohomology degree. (Recall that the cohomology
  degree is double the polynomial degree.)  

  Note that the degree of $p_i$ is $1$, so by
  Proposition~\ref{prop:deg-pvA} the left hand side
  of~\eqref{eq:general-Monk} is homogeneous of degree
  $\abs{\mathcal{A}}+1$.  Since each
  $c^{\mathcal{B}}_{i,{\mathcal{A}}}$ is a polynomial in
  $\C[t]$, the term $c^{\mathcal{B}}_{i,{\mathcal{A}}}
  p_{v_{\mathcal{B}}}$ in the right hand side
  of~\eqref{eq:general-Monk} has degree at least $\abs{\mathcal{B}}$. The
  degree of the right hand side agrees with that of the left, and
  the $\{p_{v_{\mathcal A}}\}$ are $\C[t]$-linearly independent, so
  $c^{\mathcal{B}}_{i,{\mathcal{A}}} = 0$ if $\abs{\mathcal{B}} >
  \abs{\mathcal{A}}+1$.  In other words
\[
p_i \cdot p_{v_{\mathcal{A}}} = \sum_{|{\mathcal{B}}| \leq
  |{\mathcal{A}}|+1} c^{\mathcal{B}}_{i,{\mathcal{A}}} \cdot  p_{v_{\mathcal{B}}}.
\]

We now claim that $c^{\mathcal{B}}_{i,\mathcal{A}}=0$ for any
$\mathcal{B}$ with $\mathcal{A} \not \subseteq \mathcal{B}$. We argue
by contradiction. Suppose there exists some $\mathcal{B}$ with $\mathcal{B} \not \supseteq \mathcal{A}$ 
such that $c^{\mathcal{B}}_{i,\mathcal{A}} \neq 0$. Then there is a
minimal such; denote it
$\mathcal{A}'$. Evaluate the equation~\eqref{eq:general-Monk} at
$w_{\mathcal{A}'}$.  Since $\mathcal{A} \not \subseteq \mathcal{A}'$
the left hand side is zero.  The minimality assumption on
$\mathcal{A}'$ implies that $c_{i,\mathcal{A}}^{\mathcal{B}} = 0$ if $\mathcal{B} \subsetneq \mathcal{A}'$ while the upper-triangularity property of Peterson
Schubert classes implies that $p_{v_{\mathcal{B}}}(w_{\mathcal{A}'}) = 0$ if $\mathcal{B} \not \subseteq \mathcal{A}'$.  Hence
\[
0 = c^{\mathcal{A}'}_{i,\mathcal{A}} \cdot p_{v_{\mathcal{A}'}}(w_{\mathcal{A}'}).
\]
Since $H^*_{S^1}(\pt) \cong \C[t]$ is an integral domain, either $c^{\mathcal{A}'}_{i,\mathcal{A}}$ or $p_{v_{\mathcal{A}'}}(w_{\mathcal{A}'})$ is zero.   By Corollary~\ref{cor:pvAwA}  we
conclude $c^{\mathcal{B}}_{i,\mathcal{A}} = 0$
if $\mathcal{A} \not \subseteq \mathcal{B}$.  This proves~\eqref{eq:Monk-refined}. 

To prove the formula for $c^{\mathcal{A}}_{i,\mathcal{A}}$ we 
evaluate~\eqref{eq:Monk-refined} at the fixed point $w_{\mathcal{A}}$. If ${\mathcal{B}}$ satisfies ${\mathcal{A}} \subsetneq {\mathcal{B}}$ then $w_{\mathcal{B}}  >w_{\mathcal{A}}$ so $p_{v_{\mathcal{B}}}(w_{\mathcal{A}}) = 0$. We conclude 
\[
p_i(w_{\mathcal{A}}) p_{v_{\mathcal{A}}}(w_{\mathcal{A}}) =
c^{\mathcal{A}}_{i,{\mathcal{A}}} \cdot  p_{v_{\mathcal{A}}}(w_{\mathcal{A}}).
\]
Since $p_{v_{\mathcal{A}}}(w_{\mathcal{A}}) \neq 0$ and
$H^*_{S^1}(\pt) \cong \C[t]$ is an integral domain, we conclude
\[
c^{\mathcal{A}}_{i,{\mathcal{A}}} = p_i(w_{\mathcal{A}}).
\]

To prove the last claim, suppose that $\mathcal{B}$ is such that
\(\mathcal{A} \subsetneq \mathcal{B}\) and \(\abs{\mathcal{B}} =
\abs{\mathcal{A}}+1.\) 
Evaluating~\eqref{eq:Monk-refined} at the fixed point $w_{\mathcal{B}}$
we obtain
\[
p_i(w_{\mathcal{B}}) \cdot p_{v_{\mathcal{A}}}(w_{\mathcal{B}}) =
c^{\mathcal{A}}_{i,{\mathcal{A}}} \cdot p_{v_{\mathcal{A}}}(w_{\mathcal{B}})
+ \sum_{{\mathcal{A}} \subsetneq {\mathcal{B}} \textup{ and } \abs{\mathcal{B}}
  = \abs{\mathcal{A}}+1}   c^{\mathcal{B}}_{i,{\mathcal{A}}} \cdot p_{v_{\mathcal{B}}}(w_{\mathcal{B}}).
\]
The previous claim showed $c^{\mathcal{A}}_{i,{\mathcal{A}}} =
p_i(w_{\mathcal{A}})$.  If $\mathcal{B}' \neq \mathcal{B}$ is
another subset in the sum above, 
the upper-triangularity
condition on the Peterson Schubert classes implies
$p_{v_{\mathcal{B}'}}(w_{\mathcal{B}}) = 0$. Hence  
\[
p_i(w_{\mathcal{B}}) \cdot p_{v_{\mathcal{A}}}(w_{\mathcal{B}}) = 
c^{\mathcal{A}}_{i,\mathcal{A}} \cdot p_{v_{\mathcal{A}}}(w_{\mathcal{B}}) +
c^{\mathcal{B}}_{i,\mathcal{A}} \cdot p_{v_{\mathcal{B}}}(w_{\mathcal{B}}).
\]
By Corollary~\ref{cor:pvAwA}, we know
$p_{v_{\mathcal{B}}}(w_{\mathcal{B}}) \neq 0$, so we may solve for 
$c^{\mathcal{B}}_{i,{\mathcal{A}}}$ to obtain~\eqref{eq:cBiA-formula},
as desired. 
\end{proof}

Next we compute explicitly the 
expression for $c^{\mathcal{B}}_{i,{\mathcal{A}}}$ in~\eqref{eq:cBiA-formula}.
We need some preliminary lemmas.

\begin{lemma}
Suppose $i\in \{1,2,\ldots,n-1\}$ and ${\mathcal{A}} \subseteq \{1,2,\ldots,n-1\}$.
\begin{itemize} 
\item If $i \not \in {\mathcal{A}}$ then $p_i(w_{\mathcal{A}}) = 0$.  
\item If $i \in \mathcal{A}$ then 
\begin{equation}\label{eq:piwA}
p_i(w_{\mathcal{A}}) =
\left(\head{{\mathcal{A}}}{i}-i+1\right)\left(i-\toe{{\mathcal{A}}}{i}+1\right)t.
\end{equation}
\end{itemize} 
\end{lemma}

\begin{proof}
  If $i$ is not contained in ${\mathcal{A}}$ then $s_i$ does not
  appear in $w_{\mathcal{A}}$ and so $p_i(w_{\mathcal{A}}) = 0$. Now
  suppose \(i \in {\mathcal{A}}.\) We saw in
  Lemma~\ref{lemma:length-of-roots} that each summand in
Billey's formula for 
  $p_i(w_{\mathcal{A}})$ is \((i -
  \toe{{\mathcal{A}}}{i} +1) t.\) On the other hand $s_i$ appears
  exactly \(\head{{\mathcal{A}}}{i}-i+1\) times in the reduced word
  for $w_{\mathcal{A}}$ given in
  equation~\eqref{eq:wA-reduced-word}, by inspection. 
Equation~\eqref{eq:piwA} now follows from 
  Proposition~\ref{prop:number-of-ways}.
\end{proof}

The previous lemma lets us further refine 
the vanishing conditions for $c^{\mathcal{B}}_{i,{\mathcal{A}}}$. We begin with terminology. 

\begin{definition}
Given any index $k$ and any
subset ${\mathcal{A}} \subseteq \{1,2,\ldots, n-1\}$ containing $k$,
we refer to $[\toe{\mathcal{A}}{k}, \head{\mathcal{A}}{k}]$ as 
{\bf the maximal consecutive substring of ${\mathcal{A}}$ which contains
  $k$}. 
\end{definition}

Let ${\mathcal{A}} \subseteq \{1,2,\ldots,n-1\}$. 
If ${\mathcal{B}}$ is a
subset such that ${\mathcal{A}} \subsetneq {\mathcal{B}}$ and
$|{\mathcal{B}}|=|{\mathcal{A}}|+1$ then there exists $k \in
\{1,2,\ldots,n-1\}$ with \(k \not \in \mathcal{A}\) and 
${\mathcal{B}} = {\mathcal{A}} \cup \{k\}$.  
Exactly one of the following occurs:
\begin{enumerate} 
\item a maximal consecutive substring in ${\mathcal{A}}$ is
  lengthened, from either \([k+1, \head{{\mathcal{B}}}{k}]\) or
  \([\toe{{\mathcal{B}}}{k}, k-1]\) to $[\toe{{\mathcal{B}}}{k},
  \head{{\mathcal{B}}}{k}]$, with either $\toe{\mathcal{B}}{k} = k$ or
$\head{\mathcal{B}}{k}=k$ respectively;
\item two maximal consecutive substrings in ${\mathcal{A}}$ are
  merged, namely $[\toe{{\mathcal{B}}}{k},k-1]$ and $[k+1,
  \head{{\mathcal{B}}}{k}]$ are both in ${\mathcal{A}}$ and become
  $[\toe{{\mathcal{B}}}{k}, \head{{\mathcal{B}}}{k}]$ in $\mathcal{B}$; or
\item the new index $k$ is itself a maximal consecutive
  substring $\{k\} = [k,k]$ in ${\mathcal{B}}$. 
\end{enumerate} 
Conversely, all but one of the
maximal consecutive strings of ${\mathcal{B}}$ is a maximal consecutive
string in ${\mathcal{A}}$. Summarizing, the maximal consecutive strings of ${\mathcal{A}}$ that differ from the maximal consecutive strings in ${\mathcal{B}}$ are 
\[
\{\toe{{\mathcal{B}}}{k}, \toe{{\mathcal{B}}}{k}+1, \ldots, k-1\} \subseteq {\mathcal{A}} \quad \mbox{and} \quad 
\{k+1, k+2, \ldots, \head{{\mathcal{B}}}{k}\} \subseteq {\mathcal{A}},
\]
of lengths 
\[
k-1-\toe{{\mathcal{B}}}{k}+1 = k-\toe{{\mathcal{B}}}{k} \quad \mbox{and} \quad \head{{\mathcal{B}}}{k}-k-1+1=\head{{\mathcal{B}}}{k}-k
\]
respectively. (The first string is empty if \(k=\toe{{\mathcal{B}}}{k}\) and the second string is empty if \(\head{{\mathcal{B}}}{k}=k.\))

\begin{lemma} \label{lemma:vanishing-for-pi}
Suppose ${\mathcal{B}}$ is the disjoint union ${\mathcal{B}} = {\mathcal{A}} \cup \{k\}$.
If either one of the following conditions hold: 
\begin{itemize}
\item \(i \not \in {\mathcal{B}},\) or 
\item \(i \in {\mathcal{B}},\) and $i$ and $k$ are not contained in the same maximal consecutive substring in ${\mathcal{B}}$, namely \(\toe{{\mathcal{B}}}{i} = \toe{{\mathcal{A}}}{i}\) and \(\head{{\mathcal{B}}}{i} = \head{{\mathcal{A}}}{i},\)
\end{itemize}
then $c^{\mathcal{B}}_{i,{\mathcal{A}}} = 0$. 
\end{lemma} 

\begin{proof} 
In the first case $i \not \in {\mathcal{B}}$ and so $i \not \in {\mathcal{A}}$; hence both $p_i(w_{\mathcal{B}}) = 0$ and $p_i(w_{\mathcal{A}}) = 0$.
In the second case $p_i(w_{\mathcal{A}}) = p_i(w_{\mathcal{B}})$.
The claim now follows from Equation~\eqref{eq:cBiA-formula}. 
\end{proof}

The above lemma suggests that the information needed to compute
$c^{\mathcal{B}}_{i,\mathcal{A}}$ is contained in the
maximal consecutive substring of $\mathcal{B}$ containing $i$, and
that we should be able to ``ignore'' all other maximal consecutive
substrings of $\mathcal{B}$.  The next two lemmas make this idea
precise. We call two disjoint consecutive strings
\textbf{adjacent} if their union is again a consective string. The
next lemma asserts that if two disjoint subsets $\mathcal{B},
\mathcal{B}'$ contain no adjacent maximal consecutive substrings, then
the Peterson Schubert class corresponding to
$\mathcal{B} \cup \mathcal{B}'$ is simply the product of the
classes corresponding to $\mathcal{B}$ and $\mathcal{B}'$ respectively.

\begin{lemma}\label{disjoint product}
Let ${\mathcal{B}}$ and ${\mathcal{B}}'$ be disjoint subsets of
$\{1,2,\ldots,n-1\}$. Suppose that $\mathcal{B}$ and $\mathcal{B}'$ contain
{\em no} adjacent maximal consecutive substrings, i.e. there
exists 
no $j \in {\mathcal{B}}, j' \in {\mathcal{B}}'$ such that $\abs{j-j'}
= 1$. Then 
\begin{equation}\label{eq:pvBcupBprime}
p_{v_{{\mathcal{B}} \cup {\mathcal{B}}'}} =
p_{v_{\mathcal{B}}}p_{v_{{\mathcal{B}}'}}. 
\end{equation}
\end{lemma}

\begin{proof}
We prove that for all $\mathcal{A} \subseteq \{1,2,\ldots,n-1\}$ the restrictions
in~\eqref{eq:pvBcupBprime} agree at $w_{\mathcal{A}}$:
\[p_{v_{{\mathcal{B}} \cup {\mathcal{B}}'}}(w_{\mathcal{A}}) =
p_{v_{\mathcal{B}}}(w_{\mathcal{A}}) \cdot p_{v_{{\mathcal{B}}'}}(w_{\mathcal{A}}). 
 \]
We take cases. Suppose
  $\mathcal{B} \cup \mathcal{B}' \not \subseteq \mathcal{A}$, which
  implies $\mathcal{B} \not \subseteq \mathcal{A}$ or
  $\mathcal{B}' \not \subseteq \mathcal{A}$. By the
  upper-triangularity property of Peterson Schubert classes, both the
  right and left sides of Equation~\eqref{eq:pvBcupBprime} are zero.
  Hence the equality holds. 

Now let $\mathcal{B} \cup \mathcal{B}' \subseteq \mathcal{A}$. 
 If $[b_i,b_{i+1}] \subseteq {\mathcal{B}} \cup {\mathcal{B}}'$
    is a maximal consecutive substring then every reduced word for 
    $v_{{\mathcal{B}} \cup {\mathcal{B}}'}$ contains 
$v_{[b_i, b_{i+1}]}$ as a reduced subword by 
definition of  $v_{{\mathcal{B}} \cup {\mathcal{B}}'}$.  
(Fact~\ref{fact:vA-reduced-word-unique} tells us that there is a unique reduced
word for $v_{[b_i, b_{i+1}]}$.)  No two distinct maximal consecutive strings
 $[b_i,b_{i+1}]$ and $[b_j, b_{j+1}]$ are adjacent in ${\mathcal{B}} \cup {\mathcal{B}}'$
 so all simple transpositions in $v_{[b_i, b_{i+1}]}$ commute with 
 all simple transpositions in $v_{[b_j, b_{j+1}]}$.  Comparing lengths of the permutations,
 we conclude that each reduced word for $v_{{\mathcal{B}} \cup {\mathcal{B}}'}$
 can be partitioned into a unique subword that equals  $v_{\mathcal{B}}$ and a unique 
 subword that equals $v_{{\mathcal{B}'}}$.
 
Let $\mathbf{b}$ be the reduced word for $w_{{\mathcal{A}}}$ given
  in~\eqref{eq:wA-reduced-word}.  The previous discussion implies that
  $b_{j_1} b_{j_2} \cdots b_{j_{|{\mathcal{B}}| + |{\mathcal{B}}'|}}$
  is a reduced subword of $\mathbf{b}$ that equals $v_{{\mathcal{B}}
    \cup {\mathcal{B}}'}$  if and only if  
    $b_{j_1} b_{j_2} \cdots b_{j_{|{\mathcal{B}}| + |{\mathcal{B}}'|}}$ contains exactly one subword
$b_{k_1} b_{k_2} \cdots b_{k_{|{\mathcal{B}}|}}$ that equals
$v_{\mathcal{B}}$ and exactly one subword $b_{k_1'} b_{k_2'} \cdots
b_{k_{|{\mathcal{B}}'|}'}$ that equals $v_{{\mathcal{B}}'}$. 
Conversely, the product (in the ordering induced from
$\mathbf{b}$) of each pair of reduced subwords $b_{k_1} b_{k_2} \cdots
b_{k_{|{\mathcal{B}}|}} = v_{\mathcal{B}}$ and $b_{k_1'} b_{k_2'}
\cdots b_{k_{|{\mathcal{B}}'|}'} = v_{{\mathcal{B}}'}$ of $\mathbf{b}$
is a reduced subword $b_{j_1} b_{j_2} \cdots
b_{j_{|{\mathcal{B}}| + |{\mathcal{B}}'|}}$ of $\mathbf{b}$ equalling
$v_{{\mathcal{B}} \cup {\mathcal{B}}'}$.  
This implies that the number of terms in Billey's formula for
$p_{v_{{\mathcal{B}} \cup {\mathcal{B}}'}}(w_{\mathcal{A}})$ is
precisely the product of the number of terms in Billey's formula for
$p_{v_{\mathcal{B}}}(w_{\mathcal{A}})$ and
$p_{v_{{\mathcal{B}}'}}(w_{\mathcal{A}})$.  By
Proposition~\ref{prop:number-of-ways}, we need only show that each summand in 
Billey's formula for $p_{v_{{\mathcal{B}} \cup {\mathcal{B}}'}}(w_{\mathcal{A}})$ is the product 
of a summand in Billey's formula for $p_{v_{\mathcal{B}}}(w_{\mathcal{A}})$ and another for
$p_{v_{\mathcal{B}'}}(w_{\mathcal{A}})$. 

Using Lemma~\ref{lemma:length-of-roots} and the above discussion, we conclude
that the summand in Billey's formula for $p_{v_{{\mathcal{B}} \cup {\mathcal{B}}'}}(w_{\mathcal{A}})$ corresponding to $b_{j_1} b_{j_2} \cdots b_{j_{|{\mathcal{B}}| + |{\mathcal{B}}'|}}$ is 
\begin{equation*}
\prod_{i=1}^{|{\mathcal{B}}|+|{\mathcal{B}}'|}
\pi_{S^1}(r({j_i},\mathbf{b})) = \prod_{i \in {\mathcal{B}} \cup
  {\mathcal{B}}'} (i-\toe{{\mathcal{B}} \cup {\mathcal{B}}'}{i} + 1).
\end{equation*}
Since $\mathcal{B}, \mathcal{B}'$ contain no adjacent maximal
consecutive strings, for any \(i \in \mathcal{B} \cup \mathcal{B'},\)
precisely one of the following hold: either \(i \in \mathcal{B}\) and
\(\toe{\mathcal{B} \cup \mathcal{B}'}{i} = \toe{\mathcal{B}}{i}\) or
\(i \in {\mathcal{B}}'\) and $\toe{{\mathcal{B}} \cup
  {\mathcal{B}}'}{i} = \toe{{\mathcal{B}}'}{i}$.  Hence we may compute 
\begin{equation*}
\begin{split}
\prod_{i \in {\mathcal{B}} \cup
  {\mathcal{B}}'} (i-\toe{{\mathcal{B}} \cup {\mathcal{B}}'}{i} + 1)
&= \prod_{i \in {\mathcal{B}}} (i-\toe{{\mathcal{B}}}{i} + 1)\prod_{i \in
  {\mathcal{B}}'} (i-\toe{{\mathcal{B}}'}{i} + 1) \\
&= \prod_{i=1}^{|{\mathcal{B}}|} \pi_{S^1}(r({k_i},\mathbf{b}))
\prod_{i=1}^{|{\mathcal{B}}'|} \pi_{S^1}(r(k_i',\mathbf{b})). 
\end{split}
\end{equation*} 
Hence each summand in Billey's formula for the left
side of~\eqref{eq:pvBcupBprime} may be written as a product of a
summand in Billey's formula for
$p_{v_{\mathcal{B}}}(w_{\mathcal{A}})$ and another for
$p_{v_{\mathcal{B}'}}(w_{\mathcal{A}})$. The claim follows.
\end{proof}

As observed in Section~\ref{subsec:torus-fixed-pts}, any subset of
$\{1,2,\ldots,n-1\}$ decomposes into a series of
non-adjacent maximal consecutive substrings.  The above lemma
indicates that the Peterson Schubert class
associated to each set $\mathcal{A}$ may be computed in terms of the classes
corresponding to its maximal consecutive substrings.  This allows us
to derive the following simplification of one of the expressions
appearing in Equation~\eqref{eq:cBiA-formula}.

\begin{lemma}
  Suppose ${\mathcal{B}} \subseteq \{1,2,\ldots,n-1\}$ is a disjoint
  union ${\mathcal{A}} \cup \{k\}$.  Then
 \[
\frac{p_{v_{\mathcal{A}}}(w_{\mathcal{B}})}{p_{v_{\mathcal{B}}}(w_{\mathcal{B}})}
=  \frac{p_{v_{[\toe{{\mathcal{B}}}{k}, k-1]}}(w_{\mathcal{B}})
  p_{v_{[k+1,\head{{\mathcal{B}}}{k}]}}(w_{\mathcal{B}})}{p_{v_{[\toe{{\mathcal{B}}}{k},
      \head{{\mathcal{B}}}{k}]}}(w_{\mathcal{B}})} =
\frac{p_{v_{[\toe{{\mathcal{B}}}{k}, k-1] \cup
      [k+1,\head{{\mathcal{B}}}{k}]}}(w_{\mathcal{B}})}{p_{v_{[\toe{{\mathcal{B}}}{k},
      \head{{\mathcal{B}}}{k}]}}(w_{\mathcal{B}})}.
\]
\end{lemma}

\begin{proof}
Suppose that ${\mathcal{A}}$ decomposes into maximal consecutive substrings as
\[
{\mathcal{A}} = [\toe{{\mathcal{B}}}{k}, k-1] \cup [k+1,\head{{\mathcal{B}}}{k}] \cup  [a_1, a_2] \cup [a_3, a_4] \cup \cdots \cup [a_{m-1},a_m]
\]
and that ${\mathcal{B}}$ decomposes into maximal consecutive substrings as
\[
{\mathcal{B}} = [\toe{{\mathcal{B}}}{k}, \head{{\mathcal{B}}}{k}] \cup [a_1, a_2] \cup [a_3, a_4] \cup \cdots \cup [a_{m-1},a_m].
\]
The previous lemma showed that
\[
p_{v_{\mathcal{B}}}(w_{\mathcal{B}}) = p_{v_{[\toe{{\mathcal{B}}}{k},
    \head{{\mathcal{B}}}{k}]}}(w_{\mathcal{B}}) \cdot p_{v_{[a_1,
    a_2]}}(w_{\mathcal{B}}) \cdot p_{v_{[a_3, a_4]}}(w_{\mathcal{B}}) \cdots p_{v_{[a_{m-1}, a_m]}}(w_{\mathcal{B}})
\]
and similarly for $p_{v_{\mathcal{A}}}$.  The claim follows.
\end{proof}

As a consequence of the above, for the purposes of computing the right
hand side of Equation~\eqref{eq:cBiA-formula}, 
we may assume without loss of generality that $\mathcal{B}$ consists
of a single consecutive string $[\toe{\mathcal{B}}{k},
\head{\mathcal{B}}{k}]$ 
and $\mathcal{A} = \mathcal{B} \setminus \{k\}$) for any $k \in \mathcal{B}$. 
We can now give a combinatorial and explicit
expression for both factors in Equation~\eqref{eq:cBiA-formula}.

\begin{lemma}\label{lemma:quotient-of-pv}
Let ${\mathcal{A}} = [\toe{{\mathcal{B}}}{k}, k-1] \cup [k+1,\head{{\mathcal{B}}}{k}]$ and ${\mathcal{B}} =  [\toe{{\mathcal{B}}}{k}, \head{{\mathcal{B}}}{k}]$.
Then 
\[
p_{v_{\mathcal{A}}}(w_{\mathcal{B}}) = \frac{ (\head{{\mathcal{B}}}{k}-\toe{{\mathcal{B}}}{k}+1)!}{k-\toe{{\mathcal{B}}}{k}+1}
\binom{\head{\mathcal{B}}{k} - \toe{\mathcal{B}}{k}+1}{k-\toe{\mathcal{B}}{k}}
t^{\head{{\mathcal{B}}}{k}-\toe{{\mathcal{B}}}{k}}.
\]
In particular, 
\begin{equation}\label{eq:pvAwB-over-pvBwB}
\frac{p_{v_{\mathcal{A}}}(w_{\mathcal{B}})}{p_{v_{\mathcal{B}}}(w_{\mathcal{B}})}
= \frac{1}{k-\toe{{\mathcal{B}}}{k}+1} \cdot
\binom{\head{\mathcal{B}}{k} - \toe{\mathcal{B}}{k}+1}{k-\toe{\mathcal{B}}{k}}
\frac{1}{t}.
\end{equation}
\end{lemma}

\begin{proof} 
We apply Billey's formula to compute
$p_{v_{\mathcal{A}}}(w_{\mathcal{B}})$. 
Recall that
\[
v_{\mathcal{B}} := s_{\toe{{\mathcal{B}}}{k}} s_{\toe{{\mathcal{B}}}{k}+1} \cdots s_{\head{{\mathcal{B}}}{k}} \textup{  and  } 
v_{\mathcal{A}} := s_{\toe{{\mathcal{B}}}{k}} s_{\toe{{\mathcal{B}}}{k}+1} \cdots s_{k-1} \widehat{s_k} s_{k+1} \cdots s_{\head{{\mathcal{B}}}{k}}. 
\]
By Lemma~\ref{lemma:summands-equal}, we conclude that each
summand in Billey's formula for $p_{v_{\mathcal{A}}}(w_{\mathcal{B}})$ equals 
\[
\frac{(\head{{\mathcal{B}}}{k}-\toe{{\mathcal{B}}}{k}+1)!}{k-\toe{{\mathcal{B}}}{k}+1}
\hspace{1.5mm} t^{\head{{\mathcal{B}}}{k}-\toe{{\mathcal{B}}}{k}}. 
\]

By Proposition~\ref{prop:number-of-ways}, we need next to compute the
number of distinct ways that $v_{\mathcal{A}}$ appears as a reduced subword of
$w_{\mathcal{B}}$. First, by construction, the
element $v_{\mathcal{A}}$ is equal to 
\begin{equation}\label{eq:twofactors}
v_{[\toe{{\mathcal{B}}}{k}, k-1]} \cdot v_{[k+1, \head{{\mathcal{B}}}{k}]}.
\end{equation}
(By definition $v_{\emptyset} = 1$.) 
Moreover, both factors appear in
every reduced-word decomposition of $v_{\mathcal{A}}$ and each factor
has a \textit{unique} reduced word decomposition
(see Fact~\ref{fact:vA-reduced-word-unique}),
in which the indices in $\{\toe{{\mathcal{B}}}{k}, \ldots, k-1\}$ are listed in increasing order, as
are the indices in $\{k+1, \ldots, \head{{\mathcal{B}}}{k}\}$. 
Since the two factors correspond to non-adjacent maximal consecutive
strings, each simple transposition appearing in
$v_{[\toe{{\mathcal{B}}}{k}, k-1]}$ commutes with each such of $v_{[k+1,
 \head{{\mathcal{B}}}{k}]}$. Hence the set of reduced-word
decompositions of $v_{\mathcal{A}}$ are in bijective correspondence with orderings
of the set ${\mathcal{A}}$ such that the elements $\{\toe{{\mathcal{B}}}{k}, \ldots,
k-1\}$ appear in increasing order, as do the elements $\{k+1, \ldots,
\head{{\mathcal{B}}}{k}\}$. 

Let $\mathbf{b}$ be the reduced word decomposition for
$w_{\mathcal{B}}$ given by~\eqref{eq:wA-reduced-word}. We wish to find
subwords of $\mathbf{b}$ 
which equal $v_{\mathcal{A}}$. The
index $\head{{\mathcal{B}}}{k}$ appears only once, and as
observed above, the
indices $\{k+1, \ldots, \head{{\mathcal{B}}}{k}\}$ must appear in
increasing order.  We conclude 
that there is only one subword of $\mathbf{b}$ which 
equals \(v_{\{k+1, \ldots, \head{{\mathcal{B}}}{k}\}}.\) If
$k=\toe{{\mathcal{B}}}{k}$ this unique subword determines the
factorization, and the formula of the claim reduces to $1$. 
(In the special case when \(k=\head{{\mathcal{B}}}{k},\) the set
\(\{k+1, \ldots,\head{{\mathcal{B}}}{k}\}\) is empty and this
discussion is vacuous.) 

Suppose \(k > \toe{{\mathcal{B}}}{k}.\) 
Note that the indices $\{\toe{{\mathcal{B}}}{k}, \ldots, k-1\}$ appear in the first $\head{{\mathcal{B}}}{k}-k+2$ factors
of~\eqref{eq:wA-reduced-word} and no others.  A reduced word for $v_{[\toe{{\mathcal{B}}}{k}, k-1]}$ is a choice of the indices  $\{\toe{{\mathcal{B}}}{k}, \ldots, k-1\}$ from any of these factors, in increasing order; in other words, the reduced words for $v_{[\toe{{\mathcal{B}}}{k}, k-1]}$ in $w_{\mathcal{B}}$ correspond bijectively with ordered partitions of $\head{{\mathcal{B}}}{k}-k+2$ into $k-\toe{{\mathcal{B}}}{k}$ nonnegative parts.  This is given by the binomial coefficient
\[
\binom{\head{{\mathcal{B}}}{k}-\toe{{\mathcal{B}}}{k} +1}{k-\toe{{\mathcal{B}}}{k}} = \frac{(\head{{\mathcal{B}}}{k}-\toe{{\mathcal{B}}}{k}+1)!}{(k-\toe{{\mathcal{B}}}{k})! (\head{{\mathcal{B}}}{k}-k+1)!}.
\]
By Proposition~\ref{prop:number-of-ways} we conclude 
\[
p_{v_{\mathcal{A}}}(w_{\mathcal{B}}) = \frac{ (\head{{\mathcal{B}}}{k}-\toe{{\mathcal{B}}}{k}+1)!}{k-\toe{{\mathcal{B}}}{k}+1}
\binom{\head{\mathcal{B}}{k} - \toe{\mathcal{B}}{k}+1}{k-\toe{\mathcal{B}}{k}}
t^{\head{{\mathcal{B}}}{k}-\toe{{\mathcal{B}}}{k}}, 
\]
as desired. Formula~\eqref{eq:pvAwB-over-pvBwB} follows
immediately from the above equality and Corollary~\ref{cor:pvAwA}. 
\end{proof} 

\begin{remark}\label{remark:young}
As may be seen from the proof of the lemma above, the number of 
  distinct subwords of the reduced word decomposition $\mathbf{b}$ of
  $w_{\mathcal{B}}$ given in~\eqref{eq:wA-reduced-word} that equal
  $v_{\mathcal{A}}$, is also equal to the number of Young diagrams that fit
  in a \((\head{{\mathcal{B}}}{k}-k+1) \times
  (k-\toe{{\mathcal{B}}}{k})\) box.  We do not know whether this is
  a coincidence or, given the prevalence of Young diagrams in Schubert calculus,
   intrinsic to the product structure of the ring. 
\end{remark}

We proceed with a computation of the rest of Equation~\eqref{eq:cBiA-formula}. Here 
we assume that $i$ satisfies $\toe{{\mathcal{B}}}{k} \leq i \leq
\head{{\mathcal{B}}}{k}$, since otherwise
$c^{\mathcal{B}}_{i,{\mathcal{A}}}$ vanishes by Lemma~\ref{lemma:vanishing-for-pi}.

\begin{lemma}\label{lemma:diff-of-pi}
Let ${\mathcal{A}} = [\toe{{\mathcal{B}}}{k}, k-1] \cup
[k+1,\head{{\mathcal{B}}}{k}]$ and ${\mathcal{B}} =
[\toe{{\mathcal{B}}}{k}, \head{{\mathcal{B}}}{k}]$. Let $i$ be an
index with \(\toe{\mathcal{B}}{k} \leq i \leq \head{\mathcal{B}}{k}.\) 
Then 
  \begin{equation}\label{eq:piwB-minus-piwA}
p_i(w_{\mathcal{B}}) - p_i(w_{\mathcal{A}}) = 
\begin{cases}
(\head{{\mathcal{B}}}{k}-k+1)(i-\toe{{\mathcal{B}}}{k}+1) t & \mbox{if} \; \toe{{\mathcal{B}}}{k} \leq i \leq k-1\textup{ and} \\
(\head{{\mathcal{B}}}{k}-i+1)(k-\toe{{\mathcal{B}}}{k}+1) t & \mbox{if} \; k \leq i \leq \head{{\mathcal{B}}}{k}.
\end{cases}
\end{equation}
\end{lemma}

\begin{proof}
First suppose \(\toe{{\mathcal{B}}}{k}\leq i \leq k-1.\)
Then Equation~\eqref{eq:piwA} yields 
\[
p_i(w_{\mathcal{A}}) = (i-\toe{{\mathcal{B}}}{k}+1)(k-1-i+1)t \quad \textup{ and }
\quad
p_i(w_{\mathcal{B}}) = (i-\toe{{\mathcal{B}}}{k}+1)(\head{{\mathcal{B}}}{k}-i+1)t,
\]
hence we have, as desired,
\[
p_i(w_{\mathcal{B}}) - p_i(w_{\mathcal{A}}) 
 = (\head{{\mathcal{B}}}{k}-k+1)(i -\toe{{\mathcal{B}}}{k}+1)t.
\]
Now suppose \(i=k.\) Since $k \not \in {\mathcal{A}}$ the transposition
$s_k$ never appears in $w_{\mathcal{A}}$.  Thus we have 
\[
p_k(w_{\mathcal{A}}) = 0,
\]
and we compute
\[
p_k(w_{\mathcal{B}}) = (\head{{\mathcal{B}}}{k}-k+1)(k-\toe{{\mathcal{B}}}{k}+1)t,
\]
which also agrees with Equation~\eqref{eq:piwB-minus-piwA}. Finally suppose
\(k+1 \leq i \leq \head{{\mathcal{B}}}{k}.\) Then 
\[
p_i(w_{\mathcal{A}}) = (\head{{\mathcal{B}}}{k}-i+1)(i-k-1+1)
\quad
\mbox{and}
\quad
p_i(w_{\mathcal{B}}) = (\head{{\mathcal{B}}}{k}-i+1)(i-\toe{{\mathcal{B}}}{k}+1)t,
\]
so as desired
\[
p_i(w_{\mathcal{B}})-p_i(w_{\mathcal{A}}) = 
(\head{{\mathcal{B}}}{k}-i+1)(k-\toe{{\mathcal{B}}}{k}+1)t.
\]
\end{proof}

We may now state and prove our main theorem, \textbf{the
  $S^1$-equivariant Chevalley-Monk formula
  for type $A$ Peterson varieties}, which gives a ``manifestly
positive'' combinatorial formula for the non-negative, integral
structure constants $c^{\mathcal{B}}_{i,\mathcal{A}}$. We have the
following. 

\begin{theorem}\label{theorem:Monk} ({\bf ``The $S^1$-equivariant Chevalley-Monk formula for Peterson varieties.''}) 
Let $Y$ be the Peterson variety of type $A_{n-1}$ with the
natural $S^1$-action defined by~\eqref{eq:def-circle}. For ${\mathcal A}
\subseteq \{1,2,\ldots, n-1\}$, let $v_{\mathcal{A}} \in S_n$ be the permutation
given in Definition~\ref{def:vA}, and let $p_{v_{\mathcal{A}}}$ be the
corresponding Peterson Schubert class in $H^*_{S^1}(Y)$. Let $p_i :=
p_{s_i}$ denote the Peterson Schubert class corresponding to the
singleton subset $\{i\}$. Then 
\begin{equation}\label{eq:Monk-final}
p_i \cdot p_{v_{\mathcal{A}}} = p_i(w_{\mathcal{A}}) \cdot 
p_{v_{\mathcal{A}}} +  \sum_{{\mathcal{A}} \subsetneq {\mathcal{B}} \textup{ and } |{\mathcal{B}}|
  = |{\mathcal{A}}|+1} c^{\mathcal{B}}_{i,{\mathcal{A}}} \cdot p_{v_{\mathcal{B}}},
\end{equation}
where, for a subset $\mathcal{B} \subseteq \{1,2,\ldots,n-1\}$ which is a disjoint union
\(\mathcal{B} =
  \mathcal{A} \cup \{k\},\) 
\begin{itemize} 
\item if \(i \not \in \mathcal{B}\) then
  $c^{\mathcal{B}}_{i,\mathcal{A}} = 0$, 
\item if \(i \in {\mathcal{B}}\) and \(i \not \in [\toe{\mathcal{B}}{k},
\head{\mathcal{B}}{k}],\) then $c^{\mathcal{B}}_{i,{\mathcal{A}}} =
0$, 
\item if \(k \leq i \leq \head{\mathcal{B}}{k},\) 
  then 
\begin{equation}\label{eq:cBiA-formula-part1}
c^{\mathcal{B}}_{i,{\mathcal{A}}} = (\head{{\mathcal{B}}}{k}-i+1) \cdot \left( \begin{array}{c} \head{{\mathcal{B}}}{k} - \toe{{\mathcal{B}}}{k}+1 \\
    k-\toe{{\mathcal{B}}}{k} \end{array} \right),
\end{equation}
\item if \(\toe{{\mathcal{B}}}{k} \leq i \leq k-1,\) 
\begin{equation}\label{eq:cBiA-formula-part2}
c^{\mathcal{B}}_{i,{\mathcal{A}}} 
 = (i-\toe{{\mathcal{B}}}{k}+1) \cdot \binom{\head{{\mathcal{B}}}{k}-\toe{{\mathcal{B}}}{k}+1}{k-\toe{{\mathcal{B}}}{k}+1}.
\end{equation}
\end{itemize} 
Moreover $p_i(w_{\mathcal{A}})$ as well as each $c^{\mathcal{B}}_{i,\mathcal{A}}$ is a non-negative
integer. 
\end{theorem}

\begin{proof} 
  The product $p_i \cdot p_{v_{\mathcal{A}}}$ in
  $H^*_{S^1}(Y)$ is a linear combination of the
  form~\eqref{eq:Monk-final} by
  Proposition~\ref{proposition:Monk-refined-1}. The first two claims
  about the vanishing of $c^{\mathcal{B}}_{i,\mathcal{A}}$ 
  were shown in Lemma~\ref{lemma:vanishing-for-pi}. 
The latter two claims~\eqref{eq:cBiA-formula-part1} and~\eqref{eq:cBiA-formula-part2} follow from 
straightforward computation using Lemma~\ref{lemma:quotient-of-pv} and 
Lemma~\ref{lemma:diff-of-pi}. Moreover, 
the assumptions on $i$ 
 imply that the first factor appearing in the product on the right
 hand side of~\eqref{eq:cBiA-formula-part1}
 and~\eqref{eq:cBiA-formula-part2}, respectively, is a positive
 integer. Binomial coefficients are also
 positive integers, so we conclude that
 $c^{\mathcal{B}}_{i,\mathcal{A}}$ is always a non-negative
 integer. Finally, the fact that $p_i(w_{\mathcal{A}})$ is positive
 in the sense of Graham follows from Equation~\eqref{eq:piwA},
 or from Graham-positivity of Billey's formula. The
 result follows. 
\end{proof}

We give two fully computed examples. 

\begin{example} 
Continuing Examples~\ref{ex:Billey-summand}
and~\ref{ex:Billey-sum}, suppose $n=7, \mathcal{A}=\{1,2,3,5,6\}$ and
$\mathcal{B}= \{1,2,3,4,5,6\}$.

Suppose first \(i=3.\) Then from~\eqref{eq:piwA} we immediately compute
\[
p_3(w_{\mathcal{A}}) = 3t. 
\]
In this case \(\mathcal{B} = \mathcal{A} \cup \{4\},\) so $k=4$ and
$i=3$, so we use~\eqref{eq:cBiA-formula-part2}. We conclude
that 
\[
p_3 \cdot p_{v_{\mathcal{A}}} = 3t \cdot p_{v_{\mathcal{A}}} + 45
\cdot 
p_{v_{\mathcal{B}}}, 
\]
which may also be checked directly using the computations given in
Example~\ref{ex:Billey-sum} and~\eqref{eq:cBiA-formula}.

Now suppose \(i \not \in \mathcal{A}\) but \(i \in \mathcal{B},\)
i.e. \(i=4.\) In this case \(k=i=4\) and \(i \not \in \mathcal{A},\) so we
immediately see \(p_i(w_{\mathcal{A}}) = 0.\) We also
use~\eqref{eq:cBiA-formula-part1} to obtain the formula
\[
p_4 \cdot p_{v_{\mathcal{A}}} = 60 \cdot p_{v_{\mathcal{B}}},
\]
which again may be checked explicitly using the computations in
Example~\ref{ex:Billey-sum}. 

\end{example}

We conclude with some remarks about explicit presentations of
$H^*_{S^1}(Y)$ and $H^*(Y)$ via generators and relations. 
By Proposition~\ref{prop:pi-generate}, 
the equivariant Chevalley-Monk formula above
completely 
determines the ring structure of $H^*_{S^1}(Y)$. This leads to the following.

\begin{corollary}\label{corollary:ring-presentation-eqvt}
Let $Y$ be the Peterson variety of type $A_{n-1}$ with the
natural $S^1$-action defined by~\eqref{eq:def-circle}. For ${\mathcal A}
\subseteq \{1,2,\ldots, n-1\}$, let $v_{\mathcal{A}} \in S_n$ be the permutation
given in Definition~\ref{def:vA}, and let $p_{v_{\mathcal{A}}}$ be the
corresponding Peterson Schubert class in $H^*_{S^1}(Y)$. 
Let \(t \in H^*_{S^1}(Y)\) be the image of
the generator \(t \in H^*_{S^1}(\pt) \cong \C[t].\) Then the
$S^1$-equivariant cohomology $H^*_{S^1}(Y)$ is given by 
\[
H^*_{S^1}(Y) \cong \C[t, \{p_{v_{\mathcal{A}}}\}_{\mathcal{A}
  \subseteq \{1,2,\ldots, n-1\}}] \large/ \mathcal{J}
\]
where $\mathcal{J}$ is the ideal generated by the
relations~\eqref{eq:Monk-final}.  
\end{corollary}

We next explain how the equivariant Chevalley-Monk formula of
Theorem~\ref{theorem:Monk} yields a Chevalley-Monk
formula for the \emph{ordinary} cohomology of Peterson varieties, as
well as an explicit ring presentation of $H^*(Y)$. For
this discussion we denote by $\check{\sigma}_w \in H^*(G/B)$ and
$\check{p}_w \in H^*(Y)$ the ordinary cohomology classes which are the
images of the (equivariant) Schubert and Peterson Schubert classes
under the forgetful maps $H^*_T(G/B) \to H^*(G/B)$ and $H^*_{S^1}(Y)
\to H^*(Y)$, respectively. 
We have the following. 

\begin{lemma}
The classes $\{\check{p}_{v_{\mathcal{A}}}\}_{\mathcal{A} \subseteq
  \{1,2,\ldots, n-1\}}$ form a $\C$-basis of $H^*(Y)$ and the
cohomology-degree-$2$ classes $\{\check{p}_i\}_{i=1}^{n-1}$ are a set of
ring generators of $H^*(Y)$. 
\end{lemma} 

\begin{proof} 
It is well-known that the Schubert classes $\{\check{\sigma}_w\}_{w
  \in S_n}$ form a $\C$-basis of $H^*(G/B)$ and that the
cohomology-degree-$2$ classes among the $\check{\sigma}_w$ generate the
ring $H^*(G/B)$. Carrell and Kaveh show
that the restriction map
$H^*(G/B) \to H^*(Y)$ is surjective \cite{CarKav08}, so $H^*(Y)$ is generated in degree $2$. Also, we have shown
in previous sections that $H^*_T(G/B) \to H^*_{S^1}(Y)$ is surjective, and 
that the Peterson Schubert classes form a $H^*_{S^1}(\pt)$-module basis for $H^*_{S^1}(Y)$. 
The compositions \(H^*_T(G/B) \to H^*(G/B) \to H^*(Y)\) and \(H^*_T(G/B) \to H^*_{S^1}(Y) \to H^*(Y)\) are equal, so we conclude that $H^*_{S^1}(Y) \to H^*(Y)$ is also surjective and hence 
the $\{\check{p}_{v_{\mathcal{A}}}\}$ are a $\C$-basis for
$H^*(Y)$. 
\end{proof} 

In contrast, the element \(t \in H^*_{S^1}(Y)\) given by the
image of the cohomology-degree-$2$ generator of $\C[t] \cong
H^*_{S^1}(\pt)$ lies in the kernel of the forgetful map $H^*_{S^1}(Y)
\to H^*(Y)$. This can be seen from the fact that $Y$ is the fiber of the bundle $Y \to Y
\times_{S^1} ES^1 \to BS^1$. 

From this discussion we immediately
obtain the following consequence of Theorem~\ref{theorem:Monk}. 

\begin{corollary}\label{corollary:ordinary-Monk}
 ({\bf ``The (ordinary) Chevalley-Monk formula for Peterson varieties.''}) 
Let $Y$ be the Peterson variety of type $A_{n-1}$. 
For ${\mathcal A}
\subseteq \{1,2,\ldots, n-1\}$, let $v_{\mathcal{A}} \in S_n$ be the permutation
given in Definition~\ref{def:vA}, and let $\check{p}_{v_{\mathcal{A}}}$
be the image under the forgetful map $H^*_{S^1}(Y) \to H^*(Y)$ of the 
corresponding Peterson Schubert class $p_{v_{\mathcal{A}}}$ in $H^*_{S^1}(Y)$. Let $\check{p}_i :=
\check{p}_{s_i}$ denote the class corresponding to the 
singleton subset $\{i\}$. Then 
\begin{equation}\label{eq:ordinary-Monk}
\check{p}_i \cdot \check{p}_{v_{\mathcal{A}}} = 
\sum_{{\mathcal{A}} \subsetneq {\mathcal{B}} \textup{ and } |{\mathcal{B}}|
  = |{\mathcal{A}}|+1} \check{c}^{\mathcal{B}}_{i,{\mathcal{A}}} \cdot \check{p}_{v_{\mathcal{B}}},
\end{equation}
where, for a subset $\mathcal{B} \subseteq \{1,2,\ldots,n-1\}$ which is a disjoint union
\(\mathcal{B} =
  \mathcal{A} \cup \{k\},\) the structure constant $\check{c}^{\mathcal{B}}_{i, \mathcal{A}}$ is equal to the structure constant given in Theorem~\ref{theorem:Monk}, i.e. 
\[
\check{c}^{\mathcal{B}}_{i, \mathcal{A}} = c^{\mathcal{B}}_{i, \mathcal{A}} \in \Z_{\geq 0}.
\]
In particular, each $\check{c}^{\mathcal{B}}_{i, \mathcal{A}}$ is a non-negative integer. 
\end{corollary}

\begin{proof} 
The statement is immediate from Theorem~\ref{theorem:Monk} and the
observation that $p_i(w_{\mathcal{A}})$, being a multiple of $t$, goes
to zero under the forgetful map $H^*_{S^1}(Y) \to H^*(Y)$. 
\end{proof}

Since the cohomology-degree-$2$ elements $\{\check{p}_i\}_{i=1}^{n-1}$
generate the ring, 
Corollary~\ref{corollary:ordinary-Monk} completely determines the ring
structure of $H^*(Y)$. In particular, in analogy to
Corollary~\ref{corollary:ring-presentation-eqvt}, we obtain the
following.

\begin{corollary}\label{corollary:ring-presentation-ordinary} 
  Let $Y$ be the Peterson variety of type $A_{n-1}$.
For ${\mathcal A}
\subseteq \{1,2,\ldots, n-1\}$, let $v_{\mathcal{A}} \in S_n$ be the permutation
given in Definition~\ref{def:vA}, and let $\check{p}_{v_{\mathcal{A}}}$
be the image under $H^*_{S^1}(Y) \to H^*(Y)$ of the 
corresponding Peterson Schubert class in $H^*_{S^1}(Y)$. 
Then the 
ordinary cohomology $H^*(Y)$ is given by 
\[
H^*_{S^1}(Y) \cong \C[\{\check{p}_{v_{\mathcal{A}}}\}_{\mathcal{A}
  \subseteq \{1,2,\ldots, n-1\}}] \large/ \check{\mathcal{J}}
\]
where $\check{\mathcal{J}}$ is the ideal generated by the
relations~\eqref{eq:ordinary-Monk}.  
\end{corollary}

\appendix
\section{Module bases for Borel-equivariant cohomology with field
  coefficients}

In this appendix, we state a fact (with proof) about bases for modules
over graded rings, which in particular applies to our setting of
Borel-equivariant cohomology with field coefficients. The statement is
well-known, perhaps obvious, to the experts.  However, we were unable
to find a clear reference in the literature, and include it here for
completeness, convenience, and future use.

It is known \cite{BottTu, GS} that if $H^*(X;\F)$ for a field $\F$ is
finite-dimensional and concentrated in even degree, then 
the Borel-equivariant cohomology of
$H^*_T(X;\F)$ is a free $H^*_T(\pt;\F)$-module, with a non-canonical module isomorphism 
 to the tensor product 
\[
H^*_T(X;\F) \cong H^*_T(\pt;\F) \otimes_\F H^*(X;\F).
\]
Suppose $X$ is a $T$-space such that the above holds. 
In such a situation, it is natural
to ask for a $H^*_T(\pt;\F)$-module basis for the equivariant
cohomology 
$H^*_T(X;\F)$, such that the basis elements correspond in some way to
elements of the ordinary cohomology $H^*(X;\F)$. 

Motivated by this question, we prove below a general theorem about
graded rings and modules over graded rings. One
consequence is that in many common situations in the
toric topology of algebraic varieties, the Betti
numbers of the ordinary cohomology $H^*(X;\F)$ of a $T$-space
determine the number of elements of a given degree in 
a module basis for $H^*_T(X;\F)$. 

Let $R$ be a graded ring and $M$ an $R$-module. Suppose $M$ is graded
compatibly with the $R$-module structure in the sense that \(M \cong \bigoplus_{k \geq
  0} M_k\) as additive groups and the $R$-module structure takes $R_i
\times M_k$ to $M_{i+k}$. We assume $R_0 \cong \F$. Hence, since $M$
is an $R_0$-module, it also has the structure of an $\F$-vector space,
with each $M_k$ an $\F$-subspace. Let $M_{\leq k} = \bigoplus_{j \leq
  k} M_j$ denote the subspace of $M$ consisting of graded pieces of
degree at most $k$.

\begin{proposition} \label{prop:general-module-gens}
    Let $\F$ be a field. Let $R = \bigoplus_{i \geq 0} R_i$ be a
    graded $\F$-algebra such that $R_k$ is finite-dimensional for
    all $k \geq 0$, and $R_0 \cong \F$. Let $M$ be a free finitely-generated 
    $R$-module of the form 
\[
M = R \otimes_\F V
\]
for a finite-dimensional graded $\F$-vector space
$V$, where the $R$-module structure on the right hand side is given by ordinary
multiplication on the first factor and the grading on $M$ is
given by 
\[
M_k = \bigoplus_{i+j=k} R_i \otimes_\F V_j.
\]
Suppose $\{m_{\mu,k}\}$ is a subset of $M$
satisfying 
\begin{itemize}
\item \(\deg(m_{\mu,k}) = k,\) 
\item the number of $m_{\mu,k}$ of degree $k$ is precisely
  $\dim_\F(V_k)$, and
\item the $\{m_{\mu,k}\}$ are $R$-linearly independent in $M$. 
\end{itemize}
Then the $\{m_{\mu,k}\}$ are an $R$-module basis of $M$. 
\end{proposition}

 \begin{proof}
Since the $\{m_{\mu,k}\}$ are assumed $R$-linearly independent, it
suffices to show that they $R$-span $M$. 
Let $N$ denote the $R$-submodule of $M$ generated by the
$\{m_{\mu,k}\}$. We will show that $N=M$ by proving inductively
that for each $k \geq 0$ we have 
\begin{itemize}
\item $N_{\leq k} = M_{\leq k}$ and moreover, 
\item $M_{\leq k}$ is $R$-generated by the subset $\{m_{\mu, j}: j \leq k\}$ of elements $m_{\mu,j}$ of degree less than or equal to
  $k$. 
\end{itemize}
We begin with the base case $k=0$. In this case 
\[
M_0 = R_0 \otimes_\F V_0.
\]
By assumption $R_0$ is a one-dimensional $\F$-vector space so
\(\dim_\F(M_0) = \dim_\F(V_0).\) By hypothesis there exist $\dim_\F(V_0)$
many elements $m_{\mu,0}$. These elements are assumed $R$-linearly
independent, so in particular they are $\F$-linearly
independent. Hence they $\F$-span an $\F$-subspace of $M_0$ of dimension
$\dim_\F(M_0)$, so they are a basis; we conclude $N_0 =
M_0$. We also see that $M_0$ is
$R$-generated by the $\{m_{\mu,0}\}$, as required. 

Now suppose by induction that $N_{\leq k} = M_{\leq k}$ and that
$M_{\leq k}$ is $R$-generated by the elements $\{m_{\mu,j}\}$ with
$j \leq k$. We wish to show that $N_{\leq k+1} = M_{\leq k+1}$ for
which it would suffice to show $N_{k+1} = M_{k+1}$. By definition
$N_{k+1} \subseteq M_{k+1}$, so it suffices to show \(\dim_{\F}
N_{k+1} \geq \dim_{\F} M_{k+1}.\) We first observe
that $M_{k+1}$ may be decomposed as 
\begin{equation}\label{eq:decomp}
M_{k+1} = (R_0 \otimes V_{k+1}) \bigoplus \left( \bigoplus_{\stackrel{i+j=k+1}{i>0}}
 R_i \otimes V_j \right).
\end{equation}
We first claim that any element in the second factor of this direct
sum decomposition is an $R$-linear combination of elements $m_{\mu,
  j}$ for $j \leq k$. 
Indeed, any element in $R_i
\otimes V_j$ with $i>0$ can be written as an $R$-multiple of an
element $1 \otimes V_j \in M_j$ for $j\leq k$.  By the inductive
hypothesis $M_j = N_j$ for $j \leq k$, and by definition the
$\{m_{\mu,j}\}$ for $j \leq k$ are an $R$-basis for $N_{\leq k}$. 
Multiplying an $R$-linear combination of $\{m_{\mu,j}\}$ for
$j \leq k$ by an element of $R$ is still an $R$-linear combination of
$\{m_{\mu,j}\}$ for $j \leq k$; in particular the result is still in $N$. 

We now claim the $\F$-span of the degree-$(k+1)$ elements
$\{m_{\mu,k+1}\}$ and the second factor in~\eqref{eq:decomp} is all of
$M_{k+1}$.  Note that 
\[
\text{span}_{\F} \langle m_{\mu,k+1} \rangle \cap \left(
  \bigoplus_{\stackrel{i+j=k+1}{i>0}} R_i \otimes V_j \right) = \{0\}
\]
since the $\{m_{\mu, j}\}_{j \leq k+1}$ are $R$-linearly independent
and in particular $\F$-linearly independent. Since 
$\abs{\{m_{\mu,k+1}\}} = \dim_{\F}(V_{k+1})$ and \(\text{span}_{\F}
\langle m_{\mu, k+1} \rangle \subseteq N_{k+1},\) we conclude 
\(\dim_{\F} N_{k+1} \geq \dim_{\F} M_{k+1},\) as desired.

 \end{proof}

 \begin{remark}
   We emphasize that it is crucial in this proof, as well as in the applications to
   $T$-spaces mentioned above, that we are working with vector spaces
   over a field $\F$. 
  In particular, the analogous conclusion does
   \emph{not} hold for arbitrary generalized equivariant cohomology
   theories. For instance, for 
 Borel-equivariant cohomology with
   $\Z$ coefficients, Darius Bayegan has shown via explicit
   calculation that
    the Peterson Schubert classes
   $\{p_{v_{\mathcal{A}}}\}$ in this manuscript
   are
   \emph{not} an $H^*_{S^1}(\pt;\Z)$-module basis of
   $H^*_{S^1}(Y;\Z)$, although they are a $H^*_{S^1}(\pt;\C)$-module
   basis of $H^*_{S^1}(Y;\C)$ by Theorem~\ref{theorem:pvA-basis}. 
\end{remark}

\def\cprime{$'$}

\end{document}